\documentclass[11pt]{amsart}

\usepackage{amsmath,amsfonts,amssymb,amsthm,mathrsfs,hyperref}
\usepackage{tikz}
\usetikzlibrary{matrix,arrows}
\usepackage{makecell}
\usepackage{todonotes}
\usepackage{xcolor}
\usepackage{fullpage}
\usepackage{breqn}
\usepackage{paralist}

\usepackage{arydshln}
\makeatletter
  \renewcommand*\env@matrix[1][*\c@MaxMatrixCols c]{%
    \hskip -\arraycolsep
    \let\@ifnextchar\new@ifnextchar
  \array{#1}}
\makeatother

\newcommand{\set}[1]{\left\{#1\right\}}

\def\CC{{\mathbb C}}
\def\PP{{\mathbb P}}

\newcommand{\sA}{\mathcal{A}}
\newcommand{\sV}{\mathcal{V}}
\newcommand{\sD}{\mathcal{D}}
\newcommand{\sE}{\mathcal{E}}
\newcommand{\GL}{\mathbf{GL}}
\newcommand{\ZZ}{{\mathbb Z}}

\DeclareMathOperator{\codim}{codim}
\DeclareMathOperator{\id}{id}
\DeclareMathOperator{\im}{im}
\DeclareMathOperator{\Pf}{Pf}
\DeclareMathOperator{\Sing}{Sing}
\DeclareMathOperator{\rank}{rank}

\newtheorem{theorem}{Theorem}
\newtheorem*{theorem*}{Theorem}
\newtheorem{lemma}[theorem]{Lemma}
\newtheorem{proposition}[theorem]{Proposition}
\newtheorem{corollary}[theorem]{Corollary}

\newtheorem{procedure}[theorem]{Procedure}

\theoremstyle{definition}
\newtheorem{definition}[theorem]{Definition}

\theoremstyle{remark}
\newtheorem{remark}[theorem]{Remark}

\newtheorem{example}[theorem]{Example}

\newcommand\Aut{\mathrm{Aut}}
\newcommand\PN{\PP^{N}}
\newcommand{\lin}[1]{\langle\, #1 \,\rangle}

\begin{document}

\title[Prime ideals of higher secant varieties of Veronese embeddings]{On the prime ideals of higher secant varieties of \\ Veronese embeddings of small degrees}

\author[K. Furukawa]{Katsuhisa Furukawa}
\address{Katsuhisa Furukawa, Department of Mathematics, Faculty of Science, Josai University, Saitama, Japan}
\email{katu@josai.ac.jp}

\author[K. Han]{Kangjin Han}
\address{Kangjin Han, 
Department of Liberal Arts and Sciences,
Daegu-Gyeongbuk Institute of Science \& Technology (DGIST),
Daegu 42988,
Republic of Korea}
\email{kjhan@dgist.ac.kr}

\thanks{K.F. was supported by JSPS KAKENHI Grant Number 22K03236.
  K.H. was supported by a National Research Foundation of Korea (NRF) grants funded by the Korean government (MSIT, no. 2021R1F1A104818611 and no. RS-2024-00414849).}


\begin{abstract}
  In this paper, we study minimal generators of the (saturated) defining ideal of the $k$-secant variety $\sigma_k(v_d(\PP^n))$ of the image of the $d$-uple Veronese embedding $v_d: \PP^n \rightarrow \PP^N$ with ${N=\binom{n+d}{d}-1}$, focusing on cases where the degree of $\sigma_k(v_d(\PP^n))$ is relatively small.
  First, we show that the prime ideal $I(\sigma_4(v_3(\PP^3)))$ is minimally generated by $36$ homogeneous polynomials of degree $5$.
  This implies that $\sigma_4(v_3(\PP^3)) \subset \PP^{19}$ is a del Pezzo $4$-secant variety (i.e., $\deg(\sigma_4(v_3(\PP^3))) = 105$ and the sectional genus $\pi(\sigma_4(v_3(\PP^3))) = 316$), thereby providing a new example of an arithmetically Gorenstein variety of codimension $4$.
  This result addresses the symmetric version of the ``Salmon problem'' posed by E. Allman in \cite{Allman}. As an application, we decide the non-singularity of a certain locus in $\sigma_4(v_3(\PP^3))$. Furthermore, by inheritance, we obtain the generators of $I(\sigma_4(v_3(\PP^n)))$ for all $n \geq 3$.
  
  Based on the method used for $\sigma_4(v_3(\PP^3))$, we also propose a procedure to compute the first non-trivial degree piece, $I(\sigma_k(v_d(\PP^n)))_{k+1}$, for the general $k$-secant case using prolongation and weight space decomposition. Applying this procedure, we present a few more cases of $k$-secant varieties of relatively small degrees; in each of these cases, the ideal is generated in degree $k+1$ and can be fully determined by explicitly computing all generators within this degree piece.
\end{abstract}

\maketitle
\tableofcontents \setcounter{page}{1}

\section{Introduction}

Throughout the paper, we work over $\CC$, the field of complex numbers. For a projective variety $X\subset\PN$, we define the \emph{$k$-secant variety of $X$}, denoted by $\sigma_k(X)$, as the closure of the union of all $(k-1)$-planes spanned by $k$ points on $X$. Let $v_d(\PP^n)$ denote the image of the $d$-uple Veronese embedding $v_d: \PP^n \rightarrow \PN$ with ${N=\binom{n+d}{d}-1}$, and let $\sigma_k(v_d(\PP^n)) \subset \PN$ be its $k$-secant variety.

In the cases when $n=1$ or $d=2$ (see \cite{IK}), and when $k=2$ (see \cite{Kan}), the defining equations of $\sigma_k(v_d(\PP^n))$ and the generators of its defining ideal are well studied and completely understood. A detailed list of known results for small values $k$, $d$, and $n$ can be found in \cite[\textsection{}1]{LO}. However, this problem remains largely open for general $k$, $d$, and $n$.
In general, it is quite difficult to find a \emph{new} equation of the $k$-secant variety, and much harder to determine the \emph{whole set} of equations. Knowledge of the defining equations of higher secant varieties is fundamental in the study of algebraic geometry and crucial for various applications; for instance, it can be served as a key ingredient to decide the \emph{(border) rank} or the \emph{identifiability} of a tensor (see, e.g., \cite{La}).

In this article, we present new results addressing this problem. We first focus our attention on the case of $d=3$.
For $(k,n)=(3,2)$,
the $3$-secant variety
$\sigma_3(v_3(\PP^2)) \subset \PP^{9}$ is a hypersurface of degree $4$, whose defining equation is classically known as the Aronhold invariant (see \cite{IK}, \cite[Example~1.2.1]{LO}).
For $k=3$ and $n \geq 3$,
the ideal $I(\sigma_3(v_3(\PP^n)))$ can be determined by the Aronhold and $4$-minors of the symmetric flattening matrix $\phi_{1,2}$ by inheritance \cite[Proposition~2.3.1]{LO} (see subsection~\ref{pf:thm1_and_cor2} for a brief introduction to the flattening $\phi_{a,d-a}$).
In particular, a complete set of minimal generators of $I(\sigma_3(v_3(\PP^3)))$ (the case $(k,n)=(3,3)$) is known.

When $(k,n)=(4,3)$, by \cite[Theorem~1.2.3]{LO} it was previously known only that $\sigma_4(v_3(\PP^3)) \subset \PP^{19}$ is an irreducible component of the zero set of $13$-minors of a certain Young flattening matrix, which can be viewed as a generalization of a method obtaining the Aronhold invariant.
Note that $\sigma_k(v_3(\PP^3)) = \PP^{19}$ for $k \geq 5$; hence,
the case $k=4$ yields the largest non-trivial secant variety of $v_3(\PP^3) \subset \PP^{19}$.

Our main result concerning $\sigma_4(v_3(\PP^3)) \subset \PP^{19}$ is as follows.

\begin{theorem}\label{thm_m1}
  Let $v_3: \PP^3 \rightarrow \PP^{19}$ be the triple Veronese embedding of $\PP^3$ and $S$ be the homogeneous coordinate ring of $\PP^{19}$.
  Then
  $\sigma_4(v_3(\PP^3)) \subset \PP^{19}$ is a del Pezzo $4$-secant variety,
  that is to say,
  $\deg(\sigma_4(v_3(\PP^3))) = 105$ and the sectional genus $\pi(\sigma_4(v_3(\PP^3))) = 316$.
  Furthermore,
  the prime ideal $I(\sigma_4(v_3(\PP^3)))$ is minimally generated by
  $36$ homogeneous polynomials of degree $5$,
  and its minimal free resolution is 
  \begin{gather*}
   0 \rightarrow S(-12) \rightarrow S(-7)^{36} \rightarrow S(-6)^{70} \rightarrow S(-5)^{36} \rightarrow I(\sigma_4(v_3(\PP^3))) \rightarrow 0~,
\end{gather*}
which is the one of a Gorenstein ideal of codimension $4$.

\end{theorem} 

To contextualize our theorem and characterize the secant varieties discussed in this paper, we recall the following key notions for $k$-secant varieties of relatively small degrees. Let $X \subset \PN$ be a non-degenerate variety and let $e=\codim(\sigma_k(X),\PN)$. Extending the classical degree inequality, Ciliberto and Russo \cite{CR} proved that
\begin{equation}\label{basic_inq_s_k}
\deg(\sigma_k(X))\ge  \binom{e+k}{k}~\quad \text{.}
\end{equation}
We say that $\sigma_k(X)$ is a \emph{$k$-secant variety of minimal degree} (abbreviated as \emph{$k$-VMD})
if equality holds in (\ref{basic_inq_s_k}). Furthermore, generalizing the methodology of \cite{HK}, Choe and Kwak \cite{CK} recently investigated the next-to-extremal case. In their work, $\sigma_k(X)$ is said to be a \emph{del Pezzo $k$-secant variety} (abbreviated as \emph{$k$-del Pezzo}) if $\deg \sigma_k(X) = \binom{e+k}{k}+\binom{e+k-1}{k-1}$ and the sectional genus is $\pi(\sigma_k(X)) = (k-1) \Bigl(\binom{e+k}{k}+\binom{e+k-1}{k-1}\Bigr) +1$, which is the maximum possible value in characteristic $0$ (see also \textsection{}\ref{sec:secant-vari-minim}).

Arithmetically Gorenstein varieties have been extensively studied in the literature. For codimension $2$, any Gorenstein ideal is equal to a complete intersection ideal. There is also a well-known structure theorem due to Buchsbaum and Eisenbud \cite{BE} for the case of codimension $3$. For codimension $4$, there are relatively few results on the case (see, for example, \cite{KM,R15,CGNW,W} and references therein). As shown in Theorem~\ref{thm_m1}, $\sigma_4(v_3(\PP^3))$ provides an interesting \emph{new} example of an arithmetically Gorenstein variety of codimension $4$.

In fact, the $36$ minimal generators of $I(\sigma_4(v_3(\PP^3)))$ from Theorem~\ref{thm_m1} are explicitly constructed in Theorem~\ref{FGH-prolong}. We first present three defining quintic equations $F, G, H$ of $\sigma_4(v_3(\PP^3)) \subset \PP^{19}$ in \textsection{}\ref{sub_sect_FGH}, and then construct the full set of $36$ homogeneous polynomials of degree $5$ using certain automorphisms (see Remark~\ref{symm_Salmon} for their relation to the ``Salmon problem'').

Recall that for any variety $Z\subset\PP W$ invariant under the action of an algebraic group $G$, the generators of the ideal of $Z$ will be grouped into $G$-modules. In our setting, $Z=\sigma_k(v_d(\PP V))$ with a complex vector space $V$, and $G=\GL(V)$ is the general linear group. Then the free resolution in Theorem \ref{thm_m1} can also be represented in terms of Schur modules (see Corollary \ref{res_Schur}). Moreover, by symmetric inheritance \cite[Proposition~2.3.1]{LO},
we determine the ideal of $\sigma_4(v_3(\PP^n))$ for any $n\ge3$ as follows.

\begin{corollary}\label{res_inv_form}
  Let $V=\CC^{n+1}$. For any $n \geq 3$,
  the ideal
  $I(\sigma_4(v_3(\PP V)))$ is generated in degree $5$ by the irreducible $\GL(V)$-module $S_{5,4,4,2}V$ and $5$-minors of the symmetric flattening $\phi_{1,2}$.
\end{corollary}

As an application of Theorem \ref{thm_m1}, we can show the non-singularity of a certain locus in $\sigma_4(v_3(\PP^3))$ (see also \cite[Theorem 5]{FH} and Remark~\ref{s4v3P3-smoothness-rem} below). In general, it is well-known that the $(k-1)$-secant variety is contained in the \emph{singular locus} of the $k$-secant variety. In our setting, this means that $\sigma_{3}(v_3(\PP^3)) \subset \Sing(\sigma_4(v_3(\PP^3)))$.
Using four of the $36$ quintics (see Lemma~\ref{iso-outside-xi}),
we check that $\sigma_4(v_3(\PP^3))$ is
non-singular at every point of the following sublocus.

\begin{corollary}\label{s4v3P3-smoothness}
  Let $\PP^2 \subset \PP^3$ be any $2$-plane. Then
  $\sigma_4(v_3(\PP^3)) \subset \PP^{19}$ is smooth at all points in
  $\sigma_{4}(v_3(\PP^2))\setminus\sigma_{3}(v_3(\PP^3))$.
  (Note that $\sigma_{4}(v_3(\PP^2)) = \lin{v_3(\PP^2)}$, a $9$-plane in $\PP^{19}$.)
\end{corollary}

Next, we generalize our approach and main results above to other $k$-secant varieties of Veronese embeddings by investigating what kind of a $\sigma_k(v_d(\PP^n))$ does become either a $k$-secant variety of minimal degree or a del Pezzo $k$-secant variety.
When $n=2$, the degree of $\sigma_k(v_d(\PP^2))$ is known for any $k\le 8$ and $k-1\le d$
due to \cite{ElS};
hence it is possible to determine whether $\sigma_k(v_d(\PP^2))$ in this range is of minimal degree or not.
In \cite{Ott}, it was shown that $\sigma_7(v_3(\PP^4))$ is a hypersurface of degree $15$;
in particular, it is a del Pezzo $7$-secant variety.
Theorem~\ref{thm_m1} deals with the case $(k,d,n) = (4, 3, 3)$.
In addition, by carefully considering the symmetry and combinatorial linear relations among the coefficients of any polynomial of a given weight (Proposition~\ref{coeff-f-in-ker}), for instance, we can prove the following:

\begin{theorem}\label{main_thm_more}
The ideal $I(\sigma_k(v_d(\PP^n)))$ is generated in degree $k+1$ and can be determined explicitly in case of $(k,d,n)=(3,4,2)$, $(6,5,2)$ or $(8,4,3)$.    
\end{theorem}

The proof of Theorem \ref{main_thm_more} and details are given in \textsection{}\ref{sec:some-small-degrees}. 
Table~\ref{tab:types-of-some-secant} summarizes all the cases treated in this paper, where $e$ is the codimension of
$\sigma_k(v_d(\PP^n))$ in $\PP^{N=\binom{n+d}{d}-1}$.
Some cases in Table~\ref{tab:types-of-some-secant} provide affirmative examples related to Question 1.6 in \cite{BGL} (see Remark \ref{relation_with_BGL}).
Note that the table may be extended further with reasonable computational resources by following Procedure \ref{our_procedure}.
See also Remark~\ref{rem:proc} for a discussion on the practical limitations of naive computation, the efficiency of the proposed method, and our Macaulay2 package implementation.
\begin{table}
  \begin{tabular}{lrrll}
    \hline\hline
    \((k,d,n)\) & \(N\) & \(e\) & Type of small degree & Comments or References\\
    \hline\hline
    \((2, 3, 2)\) & 9 & 4 & 2-minimal & classical, \cite{ElS}\\
    \hline
    \((3, 3, 2)\) & 9 & 1 & 3-minimal & Aronhold hypersurface, \cite{Ott}\\
    \hline
    \((2, 4, 2)\) & 14 & 9 & not 2-minimal, not 2-del Pezzo & \textsection{}\ref{sec:other-higher-secant}\\
    \hline
    \((3, 4, 2)\) & 14 & 6 & 3-del Pezzo & \textsection{}\ref{subsect_dP varieties}, \cite{CK} \\
    \hline
    \((4, 4, 2)\) & 14 & 3 & 4-minimal & \cite{ElS,CR}\\
    \hline
    \((5, 4, 2)\) & 14 & 1 & 5-minimal  & 5-defective,  Clebsch quartics cut out by $\det\phi_{2,2}$\\
    \hline
    \((5, 5, 2)\) & 20 & 6 & not 5-minimal, not 5-del Pezzo & \textsection{}\ref{sec:other-higher-secant}\\
    \hline
    \((6, 5, 2)\) & 20 & 3 & 6-del Pezzo & \textsection{}\ref{subsect_dP varieties}\\
    \hline
    \((7, 5, 2)\) & 20 & 0 & 7-minimal & \cite{CR}, fill the ambient space\\
    \hline
   \((8, 6, 2)\) & 27 & 4 & not 8-minimal, not 8-del Pezzo & \textsection{}\ref{sec:other-higher-secant}\\
    \hline
    \((9, 6, 2)\) & 27 & 1 & 9-minimal & \cite{ElS}\\
    \hline
    \((3, 3, 3)\) & 19 & 8 & not 3-minimal, not 3-del Pezzo & \textsection{}\ref{sec:other-higher-secant}\\
    \hline
    \((4, 3, 3)\) & 19 & 4 & 4-del Pezzo & Theorem~\ref{thm_m1}\\
    \hline
    \((5, 3, 3)\) & 19 & 0 & 5-minimal & fill the ambient space\\
    \hline
    \((7, 4, 3)\) & 34 & 7 & not 7-minimal, not 7-del Pezzo & \textsection{}\ref{sec:other-higher-secant}\\
    \hline
    \((8, 4, 3)\) & 34 & 3 & 8-minimal & \textsection{}\ref{subsect_k-VMD}\\
    \hline
    \((9, 4, 3)\) & 34 & 1 & 9-minimal  & 9-defective, hypersurface cut out by $\det\phi_{2,2}$\\
    \hline
    \((6, 3, 4)\) & 34 & 5 & not 6-minimal, not 6-del Pezzo & \textsection{}\ref{sec:other-higher-secant}\\
    \hline
    \((7, 3, 4)\) & 34 & 1 & 7-del Pezzo & 7-defective case, \cite{Ott}\\
    \hline
    \((8, 3, 4)\) & 34 & 0 & 8-minimal & fill the ambient space\\
    \hline
    \((14, 4, 4)\) & 69 & 1 & 14-minimal & 14-defective, hypersurface cut out by $\det\phi_{2,2}$\\
    \hline
    \((9, 3, 5)\) & 55 & 2 & not 9-minimal, not 9-del Pezzo & \textsection{}\ref{sec:other-higher-secant}\\
    \hline
  \end{tabular}
  \caption{Types of some higher secant varieties of small degree}
   \label{tab:types-of-some-secant}
\end{table}

Finally, we note that the two cases $(k,d,n) = (3, 4, 2), (4, 3, 3)$ yield del Pezzo higher secant varieties of relatively large codimension $e=6, 4$, respectively, and are special in the following sense:

\begin{remark}\label{s4v3P3-smoothness-rem}
  The two cases $(k,d,n)=(3, 4, 2)$ in Theorem~\ref{main_thm_more} and $(4, 3, 3)$ in Theorem~\ref{thm_m1} are \textit{the only two exceptions}
  in the ``trichotomy pattern'' for (non-)singularity
  of subsecant loci of higher secant varieties
  (see \cite[Remark~4(e)]{FH}).
  In fact, since
  $\sigma_4(v_3(\PP^3))$ is smooth at
  any point outside $\bigcup_{\PP^2 \subset \PP^3} \sigma_{4}(v_3(\PP^2))$, we can obtain $\Sing(\sigma_4(v_3(\PP^3)))=\sigma_3(v_3(\PP^3))$ using Corollary~\ref{s4v3P3-smoothness} (see \cite[Example~6]{FH}).
\end{remark}

The paper is organized as follows.
In \textsection{}\ref{sect_prolong},
we study the structure of prolongation on the direct sum of weight spaces with respect to the Veronese embedding $v_d: \PP^n \rightarrow \PP^{N}$, and derive a combinatorial linear relations for coefficients of polynomials of a given weight (Proposition~\ref{prop-dsum-beta} and Corollary~\ref{coeff-f-in-ker}). It turns out very useful for exactly calculating the vector space $I(\sigma_k(v_d(\PP^n)))_{k+1}$ by hand or by symbolic computations.
In \textsection{}\ref{sect_s4(v3(P3))},
we focus on $\sigma_4(v_3(\PP^3)) \subset \PP^{19}$ and the generators of its defining ideal.
In \textsection{}\ref{sub_sect_FGH}, two polynomials $F$ and $H$ are given as Pfaffians of $10 \times 10$ skew symmetric matrices, whereas the polynomial $G$ is presented in a direct manner (in \textsection{}\ref{sec:minim-free-resol} we also express $G$ by a function in terms of minors of differentials in the free resolution). Then $36$ polynomials are given by polynomials $F, G, H$ and automorphisms on the coordinate ring of $\PP^{19}$ induced from those of $\PP^3$.
In fact we show that the vector space $I(\sigma_4(v_3(\PP^3)))_5$ is spanned by these $36$ polynomials. Corollary~\ref{s4v3P3-smoothness} for smoothness of subsecant loci can be proven using these equations.
In \textsection{}\ref{sec:secant-vari-minim}, we briefly recall the notions of $k$-secant varieties of minimal degree and del Pezzo $k$-secant varieties by \cite{CR,CK} and related results. Then, applying the characterization of del Pezzo higher secant varieties of \cite{CK}, we finish the proofs of Theorem~\ref{thm_m1} and Corollary~\ref{res_inv_form} in \textsection{}\ref{pf:thm1_and_cor2}.
In \textsection{}\ref{sec:some-small-degrees}, based on the method used for $\sigma_4(v_3(\PP^3))$, we first propose an algorithm (Procedure \ref{our_procedure}) to compute the first non-trivial degree piece $I(\sigma_k(v_d(\PP^n)))_{k+1}$ for a general $k$-secant case in terms of prolongation and weight space decomposition.
Using this, we treat a few more cases of $\sigma_k(v_d(\PP^n))$'s, which turn out to be del Pezzo $k$-secant varieties (\textsection{}\ref{subsect_dP varieties}) or $k$-secant varieties of minimal degree (\textsection{}\ref{subsect_k-VMD}), as well as other cases that are neither (\textsection{}\ref{sec:other-higher-secant}). Theorem~\ref{main_thm_more} follows from Theorems \ref{case-3,4,2}, \ref{case-6,5,2}, and \ref{case-8,4,3}.
\\

\paragraph*{\textbf{Acknowledgements}} We would like to thank Xianglong Ni for letting us know Corollary~\ref{res_Schur}, the invariant form of the minimal free resolution in the main theorem and are also grateful to Giorgio Ottaviani and Luke Oeding for related useful discussions as well. Finally, we would like to thank the anonymous referees very
much for their suggestions which help us make the paper more valuable.

\section{Prolongation and Weight space decomposition}\label{sect_prolong}

We use the following notation for the $d$-uple Veronese embedding
$v_d: \PP^n \rightarrow \PN$ with $N = \binom{n+d}{d}-1$.
Let $t_0, t_1, \dots, t_n$ be the homogeneous coordinates on $\PP^n$,
and let $\set{s_\alpha}$ be the homogeneous coordinates on $\PN$ with
$\alpha = (a_0, a_1, \dots, a_n) \in \ZZ_{\geq 0}^{n+1}$
and $|\alpha| = \sum a_i = d$,
where the associated ring homomorphism 
\begin{gather}\label{eq-hom-h}
  v_d^*: S(\PN) = \CC[\set{s_\alpha}] \rightarrow S(\PP^n) = \CC[t_0, t_1, \dots, t_n]
\end{gather}
is defined by
$v_d^*(s_\alpha) = t^{\alpha} = t_0^{a_0}t_1^{a_1} \dots t_n^{a_n}$.
We take $\mathfrak{S}_{n+1}$ as the full symmetric group consisting of permutations of $n+1$ symbols $\set{0,1,\dots,n}$. For a permutation $\tau \in \mathfrak{S}_{n+1}$,
we set
\begin{gather}\label{eq:aut-phi-sigma}
  \phi_{\tau}:\CC[\{s_\alpha\}] \rightarrow \CC[\{s_\alpha\}]
\end{gather}
to be the induced automorphism
sending $s_\alpha=s_{(a_0, a_1, \dots, a_n)} \mapsto s_{\tau \alpha}:=s_{(a_{\tau(0)}, a_{\tau(1)}, \dots, a_{\tau(n)})}$.

We set $L_{\beta} \subset \CC[\set{s_\alpha}]_{\frac{|\beta|}{d}}$
to be a vector subspace
generated by $(v_d^*)^{-1}(t^{\beta})$ for each $\beta \in \ZZ_{\geq 0}^{n+1}$,
where a monomial $s_{\alpha_1}^{i_1} \dotsm\, s_{\alpha_{\mu}}^{i_{\mu}} \in L_{\beta}$
if and only if $i_1 \alpha_1+\dots+i_{\mu} \alpha_{\mu} = \beta$ in $\ZZ_{\geq 0}^{n+1}$.

For an integer $e > 0$,
we have a direct sum decomposition of $\CC[\set{s_\alpha}]_e$ as
\begin{gather*}
  \CC[\set{s_\alpha}]_e = \bigoplus_{|\beta| = de} L_{\beta}.
\end{gather*}
Note that $L_{\beta}\cap L_{\beta'}=\{0\}$ if $\beta\not=\beta'$, and 
$L_{\tau \beta} = \phi_{\tau}(L_{\beta})$
for any $\beta = (b_0, b_1, \dots, b_n) \in \ZZ_{\geq 0}^{n+1}$ and $\tau \in \mathfrak{S}_{n+1}$,
where
$\tau\beta = (b_{\tau(0)}, b_{\tau(1)}, \dots, b_{\tau(n)})$.
We say that a homogeneous polynomial $f \in \CC[\set{s_\alpha}]$
is of \emph{weight} $\beta$ over $\CC[t_0, t_1, \dots, t_n]$
if $f \in L_{\beta}$; in other words,
$v_d^*(m) \in \CC[t_0, t_1, \dots, t_n]$ is of the same weight $\beta$
for every monomial term $m$ of $f$.
In this setting, it holds that $|\beta| = d \cdot \deg(f)$.
Additionally, we say that $f$ is of \emph{weight $\beta$ up to permutations}
if $f$ is of weight $\tau\beta$ for some $\tau \in \mathfrak{S}_{n+1}$.

\medskip

Next, we recall the theory of prolongation (see \cite[\textsection{}7.5]{La} and \cite[\textsection{}1 and \textsection{}2]{SS} for more details). For a vector subspace $A\subset S^d V^\ast$,
the \emph{$p$-th prolongation $A^{(p)}$ of $A$} is defined as
\[
  A^{(p)} = \left\{\,G\in S^{d+p} V^\ast \ \middle|\ \text{$\frac{\partial^p G}{\partial \mathbf{x}^\alpha}\in A$ for any $\alpha$ with $|\alpha|=p$}\,\right\}.
\]
Then, it is well-known that $I(\sigma_k(X))_{k} = 0$ and $I(\sigma_k(X))_{k+1}=I(X)_2^{(k-1)}$ for any non-degenerate variety $X\subset \PP V$ (see \cite[Theorem 7.5.3.4]{La}).
\medskip

In our setting, $X = v_d(\PP^n) \subset \PN$.
Since $I(X) = \ker(v_d^*) \subset \CC[\set{s_\alpha}]$,
it follows that
\[
  I(\sigma_k(X))_{k+1}=\ker(v_d^*)_2^{(k-1)}.
\]
Now, let $\Psi = \Psi^{(k)}$ be the composition $\oplus v_d^*\circ \mathcal{D}$ of two linear maps:
\begin{gather*}
  \mathcal{D}\ :\
  \CC[\set{s_\alpha}]_{k+1} \;\rightarrow
  \bigoplus_{m=s_{\alpha_1}^{j_1} \dotsm\, s_{\alpha_{\mu}}^{j_{\mu}}}
  \CC[\set{s_\alpha}]_{2}
  \cdot e(m)\quad,
\end{gather*}
sending $f \in \CC[\set{s_\alpha}]_{k+1}$ to
the direct sum of $\frac{\partial^{k-1} f}{(\partial s_{\alpha_1})^{j_1} \dotsm\, (\partial s_{\alpha_{\mu}})^{j_{\mu}}}$
with $m=s_{\alpha_1}^{j_1} \dotsm\, s_{\alpha_{\mu}}^{j_{\mu}}$, and
\begin{gather*}
  \oplus v_d^*
  \ :\
  \bigoplus_{m=s_{\alpha_1}^{j_1} \dotsm\, s_{\alpha_{\mu}}^{j_{\mu}}}
  \CC[\set{s_\alpha}]_{2}
  \cdot e(m)
  \;\rightarrow
  \bigoplus_{m=s_{\alpha_1}^{j_1} \dotsm\, s_{\alpha_{\mu}}^{j_{\mu}}}
  \CC[t_0, t_1, \dots, t_n]_{2d}
  \cdot e(m)\quad,
\end{gather*}
where $m=s_{\alpha_1}^{j_1} \dotsm\, s_{\alpha_{\mu}}^{j_{\mu}}$'s are taken among all the monomials of
$\CC[\set{s_\alpha}]_{k-1}$
(i.e., all $j_p\ge0$ and $\mu \geq 1$ with $j_1+\dots+j_{\mu} = k-1$)
and $e(m)$'s are basis of free modules
$\bigoplus \CC[\set{s_\alpha}] \cdot e(m)$
or
$\bigoplus \CC[t_0, t_1, \dots, t_n] \cdot e(m)$.
Then, we obtain a useful description as follows:
\begin{equation}\label{degree_piece_as_kernel}
    I(\sigma_k(v_d(\PP^n)))_{k+1} = \ker(\Psi).
\end{equation}
For a monomial
$m = s_{\alpha_1}^{j_1} \dotsm\, s_{\alpha_{\mu}}^{j_{\mu}} \in \CC[\set{s_\alpha}]_{k-1}$,
setting
\begin{gather}\label{Psi-m-f}
  \Psi_m(f) = \Psi_m^{(k)}(f) = v_d^* \left(\frac{\partial^{k-1} f}{((\partial s_{\alpha_1})^{j_1} \dotsm\, (\partial s_{\alpha_{\mu}})^{j_{\mu}})}\right),
\end{gather}
we have $\Psi(f) = \sum_{m \in \CC[\set{s_\alpha}]_{k-1}} \Psi_m(f) \cdot e(m) \in \bigoplus_{m \in \CC[\set{s_\alpha}]_{k-1}} \CC[t_0, t_1, \dots, t_n]_{2d} \cdot e(m)$.
\medskip

In general, since the dimension of the domain $\CC[\set{s_\alpha}]_e$ is very large, it may be difficult to calculate the kernel of $\Psi$ directly.
So, now we discuss how to decompose such a large space into a direct sum of smaller subspaces behaving well under $\Psi$.

For $\beta \in \ZZ_{\geq 0}^{n+1}$, let us take a partial ordering on weights in which we denote $\beta \supset \delta$ if $\beta-\delta \in \ZZ_{\geq 0}^{n+1}$. Then, we can describe the restricted linear map $\Psi|_{L_{\beta}}$ and have a direct sum decomposition (\ref{eq:dsum-vmdeg}) as below.

\begin{proposition}\label{prop-dsum-beta}
  Let $\Psi = \Psi^{(k)}$ be as above.
  Let $\beta \in \ZZ_{\geq 0}^{n+1}$ with $|\beta| = d(k+1)$,
  and let $L_{\beta} \subset \CC[\set{s_\alpha}]_{k+1}$ as above. Let us consider
  a basis of $L_{\beta}$,
  $\{\frac{1}{i_1 ! \dotsm i_{\mu}!}\, s_{\alpha_1}^{i_1} \dotsm\, s_{\alpha_{\mu}}^{i_{\mu}}\}$
  with $i_1 \alpha_1+\dots+i_{\mu} \alpha_{\mu} = \beta$.
  Then,
  the restricted linear map $\Psi|_{L_{\beta}}$ is given by
  \begin{gather}\label{eq:Psi-basis}
  \Psi\left(\frac{1}{i_1 ! \dotsm i_{\mu}!}\, s_{\alpha_1}^{i_1} \dotsm\, s_{\alpha_{\mu}}^{i_{\mu}}\right) = \sum_{
        \delta,m} \,
    \frac{1}{(i_1-j_1) ! \dotsm (i_{\mu}-j_{\mu})!}\,
    t^{\beta-\delta} e(m)\quad,
  \end{gather}
  where the sum runs over all
  $\delta \subset \beta$ with $|\delta| = d(k-1)$
  and all
  $m=s_{\alpha_1}^{j_1} \cdots s_{\alpha_{\mu}}^{j_{\mu}} \in L_{\delta}$
  with $j_p\le i_p$ for $p = 1,\dots,\mu$.
  Note that $(i_1-j_1)! \dotsm (i_{\mu}-j_{\mu})! \in \set{1,2}$, i.e.,
  the value takes $1$ if two of $i_1-j_1, \dots, i_{\mu}-j_{\mu}$ are $1$ and others are $0$,
  or takes $2$ if one of $i_1-j_1, \dots, i_{\mu}-j_{\mu}$ is $2$ and others are $0$.
  As a result, it gives the direct sum decomposition of
  the map $\Psi = \bigoplus_{|\beta| = d(k+1)} \Psi|_{L_{\beta}}$
  and of the vector space
  \begin{gather}\label{eq:dsum-vmdeg}
    I(\sigma_k(v_d(\PP^n)))_{k+1} = \ker(\Psi)
    = \bigoplus_{|\beta| = d(k+1)} K_{\beta}
  \end{gather}
  with $K_{\beta} = K_{\beta}^{(k)} := \ker(\Psi|_{L_{\beta}})$.
  Note that $K_{\tau\beta} = \phi_{\tau}(K_{\beta})$ and
  $\dim K_{\tau\beta} = \dim K_{\beta}$ for each $\tau \in \mathfrak{S}_{n+1}$.
\end{proposition}

\begin{proof}
  Let us take
  $f={\frac{1}{i_1 ! \dotsm i_{\mu}!}\, s_{\alpha_1}^{i_1} \dotsm\, s_{\alpha_{\mu}}^{i_{\mu}}}$
  and $j_1+\dots+j_{\mu} = k-1$ with $j_p \leq i_p$ for any $p = 1, \dots, \mu$.
  Setting $\delta = j_1\alpha_1+\dots+j_{\mu}\alpha_{\mu}$ and $m = s_{\alpha_1}^{j_1}, \dots, s_{\alpha_{\mu}}^{j_{\mu}}$, we have
  \begin{align*}
    \Psi_m(f) &= 
    v_d^*\left(\frac{\partial^{k-1}f}{(\partial s_{\alpha_1})^{j_1} \dotsm\, (\partial s_{\alpha_{\mu}})^{j_{\mu}}}\right)
    =
    \frac{1}{(i_1-j_1) ! \dotsm (i_{\mu}-j_{\mu})!}\,
    v_d^*(s_{\alpha_1}^{i_1-j_1} \dotsm\, s_{\alpha_{\mu}}^{i_{\mu}-j_{\mu}}) \\
    &=
    \frac{1}{(i_1-j_1) ! \dotsm (i_{\mu}-j_{\mu})!}\,
    t^{\beta-\delta}\quad.
  \end{align*}

  Note that the conditions $|\delta| = d(k-1)$ and
  $j_1\alpha_1+\dots+j_{\mu}\alpha_{\mu} = \delta$
  implies
  $j_1+\dots+j_{\mu} = k-1$.
  In particular, $(i_1+\dots+i_{\mu})-(j_1+\dots+j_{\mu}) = 2$.

  The image
  $\Psi(L_{\beta})$ is spanned by
  $\{t^{\beta-\delta} e(s_{\alpha_1}^{j_1} \dots s_{\alpha_{\mu}}^{j_{\mu}})\}$
  with $\delta \subset \beta$
  such that
  $|\delta| = d(k-1)$ and $j_1\alpha_1+\dots+j_{\mu}\alpha_{\mu} = \delta$.
  Suppose that
  $\Psi(L_{\beta}) \cap \Psi(L_{\beta}') \neq 0$. Then we have
  $t^{\beta-\delta} e(s_{\alpha_1}^{j_1} \dots s_{\alpha_{\mu}}^{j_{\mu}}) = t^{\beta'-\delta'} e(s_{\alpha_1'}^{k_1} \dots s_{\alpha_{\nu}'}^{k_{\nu}})$,
  an equality of some two vectors of the basis.
  By definition, it follows that
  $t^{\beta-\delta} = t^{\beta'-\delta'}$ and
  $e(s_{\alpha_1}^{j_1} \dots s_{\alpha_{\mu}}^{j_{\mu}}) = e(s_{\alpha_1'}^{k_1} \dots s_{\alpha_{\nu}'}^{k_{\nu}})$.
  Since $j_1\alpha_1+\dots+j_{\mu}\alpha_{\mu} = \delta$ and $k_1\alpha_1'+\dots+k_{\nu}\alpha_{\nu}' = \delta'$, we have $\delta = \delta'$, and then $\beta = \beta'$. Therefore
  $\Psi(L_{\beta}) \cap \Psi(L_{\beta'}) = 0$ if $\beta \neq \beta'$.

  As a result, since
  $\CC[\set{s_\alpha}]_{k+1} = \bigoplus L_{\beta}$
  and $\im(\Psi) = \bigoplus \Psi(L_{\beta})$,
  we have
  $\Psi = \bigoplus \Psi|_{L_{\beta}}$ and others naturally follows from this fact.
\end{proof}

Note that Proposition \ref{prop-dsum-beta} can give a necessary and sufficient condition for any $f \in L_{\beta}$ to be an element of $K_\beta$ as imposing combinatorial linear relations on the coefficients of $f$. Let $M(L_{\epsilon})$ be the set of monomials belonging to $L_{\epsilon}$ for $\epsilon \in \ZZ_{\geq 0}^{n+1}$.

\begin{corollary}\label{coeff-f-in-ker}
  For $f \in L_{\beta} \subset \CC[\set{s_\alpha}]_{k+1}$, let $c(f; s_{\alpha_1}^{i_1} \dotsm\, s_{\alpha_{\mu}}^{i_{\mu}})$ be the 
  coefficient of $f$
  with respect to a vector
  $\frac{1}{i_1 ! \dotsm i_{\mu}!}\, s_{\alpha_1}^{i_1} \dotsm\, s_{\alpha_{\mu}}^{i_{\mu}}$
  of the basis of $L_{\beta}$ (in other words, set $f=\displaystyle\sum \frac{c(f; s_{\alpha_1}^{i_1} \dotsm\, s_{\alpha_{\mu}}^{i_{\mu}})}{i_1 ! \dotsm i_{\mu}!}s_{\alpha_1}^{i_1} \dotsm\, s_{\alpha_{\mu}}^{i_{\mu}}$).
  Then, the following two conditions are equivalent:
  \begin{enumerate}[(i)]
  \item 
    $f \in K_{\beta}=\ker(\Psi|_{L_{\beta}})$.
  \item 
    For any $\epsilon \subset \beta$ with $|\epsilon| = 2d$
    and for any $m \in M(L_{\beta-\epsilon}) \subset \CC[\set{s_\alpha}]_{k-1}$,
    it holds
    \begin{equation}\label{rel_coeff-f-in-ker}
      \sum_{q \in M(L_{\epsilon})}
      \frac{1}{v(q)}\, c(f; mq) = 0,
    \end{equation}
    where $v(q) = 1$ if $q = s_{\alpha_1}s_{\alpha_2}$ for some $\alpha_1 \neq \alpha_2$, and
    $v(q) = 2$ if $q = s_{\alpha}^2$ for some $\alpha$
    (in particular, $c(f; mq) = 0$ if $M(L_{\epsilon}) = \set{q}$, a singleton).
  \end{enumerate}
\end{corollary}

\begin{proof}
  Fix $\beta \in \ZZ_{\geq 0}^{n+1}$. Then
  $t^{\beta-\delta} e(s_{\alpha_1}^{j_1} \dotsm s_{\alpha_{\mu}}^{j_{\mu}})$'s
  are linearly independent in the vector space
  $\bigoplus \CC[t_0, t_1, \dots, t_n]_{2d} \cdot e(s_{\alpha_1}^{j_1} \dotsm s_{\alpha_{\mu}}^{j_{\mu}})$.

  From (\ref{eq:Psi-basis}),
  $\Psi(f) = 0$
  if and only if
  $\sum_{(i_1, \dots, i_{\mu}) \supset (j_1, \dots, j_{\mu})}
  \frac{1}{(i_1-j_1) ! \dotsm (i_{\mu}-j_{\mu})!} c(f; s_{\alpha_1}^{i_1} \dotsm\, s_{\alpha_{\mu}}^{i_{\mu}}) = 0$
  for any
  $\delta \subset \beta$ with $|\delta| = d(k-1)$ and for any
  $s_{\alpha_1}^{j_1} \dotsm s_{\alpha_{\mu}}^{j_{\mu}} \in L_{\delta}$.
  Considering $\epsilon = \beta-\delta$ and $q=s_{\alpha_1}^{i_1-j_1} \dotsm s_{\alpha_{\mu}}^{i_{\mu}-j_{\mu}} \in L_{\epsilon}$,
  we get the condition (ii).
\end{proof}

To show how to apply the materials developed in this section to find equations of secants, let us take the following example.

\begin{example}[Aronhold case]\label{eg_Aronhold} Consider a secant hypersurface $\sigma_3(v_3(\PP^2))\subset \PP^9$. Its defining equation is classically known as the \textit{Aronhold invariant} $A_f$. For a general $f\in Sym^3(\CC^3)$, in a basis, one can write
\begin{align*}
f= &s_{300}t_0^3+s_{030}t_1^3+s_{003}t_2^3+3s_{210}t_0^2t_1+3s_{120}t_0t_1^2+3s_{201}t_0^2t_2\\
&+3s_{102}t_0t_2^2
+3s_{021}t_1^2t_2+3s_{012}t_1t_2^2+6s_{111}t_0t_1t_2 
\end{align*} 
and in this setting one can present the classical Aronhold invariant $A_f$ as follows:
\begin{align}\label{Aronhold}
A_f&=s_{111}^4-2s_{120}s_{111}^2s_{102}+s_{120}^2s_{102}^2+s_{201}s_{111}s_{102}s_{030}-s_{210}s_{102}^2s_{030}-2s_{201}s_{111}^2s_{021}\\ \nonumber
&-s_{201}s_{120}s_{102}s_{021}+3s_{210}s_{111}s_{102}s_{021}+s_{201}^2s_{021}^2-s_{300}s_{102}s_{021}^2+3s_{201}s_{120}s_{111}s_{012}\\ \nonumber
&-2s_{210}s_{111}^2s_{012}-s_{210}s_{120}s_{102}s_{012}-s_{201}^2s_{030}s_{012}+s_{300}s_{102}s_{030}s_{012}-s_{210}s_{201}s_{021}s_{012}\\ \nonumber
&+s_{300}s_{111}s_{021}s_{012}+s_{210}^2s_{012}^2-s_{300}s_{120}s_{012}^2-s_{201}s_{120}^2s_{003}+s_{210}s_{120}s_{111}s_{003}\\ \nonumber
&+s_{210}s_{201}s_{030}s_{003}-s_{300}s_{111}s_{030}s_{003}-s_{210}^2s_{021}s_{003}+s_{300}s_{120}s_{021}s_{003}~
\end{align}
(see \cite{Ott, LO} for more details). We see that $A_f\in K_\beta$ when $\beta=(4,4,4)$, because $\Psi(A_f)=0$. In the following, we are going to show that $\dim K_{(4,4,4)}=1$ and to observe how to determine all the coefficients of $A_f$ up to scalars by using Corollary \ref{coeff-f-in-ker}.

For $v_3: \PP^2 \hookrightarrow \PP^9$, we set
$v_3^*: S(\PP^9) = \CC[s_{300}, s_{210}, \dots, s_{003}] \rightarrow S(\PP^2) = \CC[t_0, t_1, t_2]$, sending $s_{ijk}$ to $t_0^i t_1^j t_2^k$. We have a weight space decomposition $\CC[s_{300}, s_{210}, \dots, s_{003}]_4 = \bigoplus_{|\beta| = 12} L_{\beta}$ and by (\ref{eq:dsum-vmdeg}) $I(\sigma_3(v_3(\PP^2)))_{4} = \bigoplus_{|\beta| = 12} K_{\beta}$.
Recall that $A_f\in K_{(4,4,4)}\subset L_{(4,4,4)}$. Note that $L_{(4,4,4)}$ is a $25$-dimensional subspace whose basis can be given as
\begin{gather*}
\begin{array}{ccccc}
s_{300} s_{120} s_{021} s_{003} &
s_{300} s_{120} s_{012}^{2} &
s_{300} s_{111} s_{030} s_{003} &
s_{300} s_{111} s_{021} s_{012} &
s_{300} s_{102} s_{030} s_{012}
\\
s_{300} s_{102} s_{021}^{2} &
s_{210}^{2} s_{021} s_{003} &
s_{210}^{2} s_{012}^{2} &
s_{210} s_{201} s_{030} s_{003} &
s_{210} s_{201} s_{021} s_{012}
\\
s_{210} s_{120} s_{111} s_{003} &
s_{210} s_{120} s_{102} s_{012} &
s_{210} s_{111}^{2} s_{012} &
s_{210} s_{111} s_{102} s_{021} &
s_{210} s_{102}^{2} s_{030}
\\
s_{201}^{2} s_{030} s_{012} &
s_{201}^{2} s_{021}^{2} &
s_{201} s_{120}^{2} s_{003} &
s_{201} s_{120} s_{111} s_{012} &
s_{201} s_{120} s_{102} s_{021}
\\
s_{201} s_{111}^{2} s_{021} &
s_{201} s_{111} s_{102} s_{030} &
s_{120}^{2} s_{102}^{2} &
s_{120} s_{111}^{2} s_{102} &
s_{111}^{4}\quad.
\end{array}
\end{gather*}

Let us take a nonzero $f \in K_{(4,4,4)}$ and compute the relations on 25 coefficients of $f$ by applying Corollary \ref{coeff-f-in-ker}. For instance, we start from a monomial $s_{210}^{2} s_{021} s_{003} \in L_{(4,4,4)}$.
\begin{enumerate}
\item[Step 1] For a choice of $\epsilon = (4, 2, 0)$ and $m=s_{021} s_{003}\in L_{\beta-\epsilon}$, we obtain one linear relation $\frac{1}{2}c(f; s_{210}^{2} s_{021} s_{003}) + c(f; s_{300} s_{120} s_{021} s_{003}) = 0$ by (\ref{rel_coeff-f-in-ker}). For $\epsilon = (0, 2, 4)$ and $m=s_{210}^{2}$, $c(f; s_{210}^{2} s_{021} s_{003}) + \frac{1}{2}c(f; s_{210}^{2} s_{012}^{2}) = 0$. Set $c(f; s_{210}^{2} s_{021} s_{003}) = -2$.
Then, by linear relations above 
$$c(f; s_{300} s_{120} s_{021} s_{003}) = 1, c(f; s_{210}^{2} s_{012}^{2}) = 4.$$ We can decide 3 coefficients in this way.

\item[Step 2] Now, for $\epsilon = (4, 2, 0)$, same as before in Step 1, and but for different $m=s_{012}^{2}$,  we get $\frac{1}{2}c(f; s_{210}^{2} s_{012}^{2}) + c(f; s_{300} s_{120} s_{012}^{2}) = 0$. It holds $c(f; s_{300} s_{111} s_{030} s_{003}) + c(f; s_{300} s_{120} s_{021} s_{003}) = 0$ for $\epsilon = (1, 4, 1)$. For $\epsilon = (3, 0, 3)$, $c(f; s_{201} s_{120} s_{102} s_{021}) + c(f; s_{300} s_{120} s_{021} s_{003}) = 0$.
Then, we determine 3 more coefficients
$$c(f; s_{300} s_{120} s_{012}^{2}) = -2, c(f; s_{300} s_{111} s_{030} s_{003}) = c(f; s_{201} s_{120} s_{102} s_{021}) = -1.$$

\item[Step 3] Likewise, by taking $\epsilon = (3, 3, 0)$, $c(f; s_{300} s_{111} s_{030} s_{003}) + c(f; s_{210} s_{120} s_{111} s_{003}) = 0$. 

For $\epsilon = (4, 1, 1)$, $c(f; s_{300} s_{111} s_{030} s_{003}) + c(f; s_{210} s_{201} s_{030} s_{003}) = 0$. 

For $\epsilon = (1, 4, 1)$, $c(f; s_{201} s_{120} s_{102} s_{021}) + c(f; s_{201} s_{111} s_{102} s_{030}) = 0$. 

For $\epsilon = (0, 3, 3)$, $c(f; s_{300} s_{111} s_{030} s_{003}) + c(f; s_{300} s_{111} s_{021} s_{012}) = 0$. 

For $\epsilon = (1, 1, 4)$, $c(f; s_{300} s_{111} s_{030} s_{003}) + c(f; s_{300} s_{102} s_{030} s_{012}) = 0$. 

\noindent Then, from these relations we see
\begin{align*}
    &1=c(f; s_{210} s_{120} s_{111} s_{003}) = c(f; s_{210} s_{201} s_{030} s_{003}) = c(f; s_{201} s_{111} s_{102} s_{030})\\
    &= c(f; s_{300} s_{111} s_{021} s_{012}) = c(f; s_{300} s_{102} s_{030} s_{012}).
\end{align*}
We continue similarly. For $\epsilon = (3, 3, 0)$, $c(f; s_{300} s_{102} s_{030} s_{012}) + c(f; s_{210} s_{120} s_{102} s_{012}) = 0$.

For $\epsilon = (2, 4, 0)$, $\frac{1}{2}c(f; s_{201} s_{120}^{2} s_{003}) + c(f; s_{210} s_{201} s_{030} s_{003}) = 0$.

For $\epsilon = (4, 1, 1)$, $c(f; s_{210} s_{201} s_{021} s_{012}) + c(f; s_{300} s_{111} s_{021} s_{012}) = 0$.

For $\epsilon = (4, 0, 2)$, $c(f; s_{300} s_{102} s_{030} s_{012}) + \frac{1}{2}c(f; s_{201}^{2} s_{030} s_{012}) = 0$.

For $\epsilon = (3, 1, 2)$, $c(f; s_{300} s_{102} s_{030} s_{012}) + c(f; s_{210} s_{102}^{2} s_{030}) + c(f; s_{201} s_{111} s_{102} s_{030}) = 0$.

For $\epsilon = (0, 4, 2)$, $c(f; s_{300} s_{102} s_{030} s_{012}) + \frac{1}{2}c(f; s_{300} s_{102} s_{021}^{2}) = 0$.

\noindent Then,
$c(f; s_{210} s_{120} s_{102} s_{012}) = c(f; s_{210} s_{201} s_{021} s_{012}) = -1$.

$c(f; s_{201} s_{120}^{2} s_{003}) = c(f; s_{201}^{2} s_{030} s_{012}) = c(f; s_{300} s_{102} s_{021}^{2}) = -2$.

$c(f; s_{210} s_{102}^{2} s_{030}) = -c(f; s_{300} s_{102} s_{030} s_{012}) - c(f; s_{201} s_{111} s_{102} s_{030}) = -2$.
  
For $\epsilon = (2, 4, 0)$, $\frac{1}{2}c(f; s_{120}^{2} s_{102}^{2}) + c(f; s_{210} s_{102}^{2} s_{030}) = 0$.

For $\epsilon = (3, 2, 1)$, $c(f; s_{300} s_{102} s_{021}^{2}) + c(f; s_{210} s_{111} s_{102} s_{021}) + c(f; s_{201} s_{120} s_{102} s_{021}) = 0$.

For $\epsilon = (2, 3, 1)$, $c(f; s_{210} s_{201} s_{021} s_{012}) + c(f; s_{201}^{2} s_{030} s_{012}) + c(f; s_{201} s_{120} s_{111} s_{012}) = 0$.

For $\epsilon = (4, 0, 2)$, $c(f; s_{300} s_{102} s_{021}^{2}) + \frac{1}{2}c(f; s_{201}^{2} s_{021}^{2}) = 0$.

For $\epsilon = (2, 2, 2)$, we have a four terms relations 

$c(f; s_{210} s_{201} s_{021} s_{012}) + c(f; s_{210} s_{120} s_{102} s_{012}) + c(f; s_{210}^{2} s_{012}^{2}) + \frac{1}{2}c(f; s_{210} s_{111}^{2} s_{012}) = 0$.

\noindent Then, we get $c(f; s_{120}^{2} s_{102}^{2}) = 4$,
$c(f; s_{210} s_{111} s_{102} s_{021}) = c(f; s_{201} s_{120} s_{111} s_{012}) = 3$ and $c(f; s_{201}^{2} s_{021}^{2}) = 4$, $c(f; s_{210} s_{111}^{2} s_{012}) = -4$.

For $\epsilon = (2, 3, 1)$, $c(f; s_{210} s_{111} s_{102} s_{021}) + c(f; s_{120} s_{111}^{2} s_{102}) + c(f; s_{201} s_{111} s_{102} s_{030}) = 0$.

For $\epsilon = (3, 1, 2)$, $c(f; s_{201} s_{111}^{2} s_{021}) + c(f; s_{300} s_{111} s_{021} s_{012}) + c(f; s_{210} s_{111} s_{102} s_{021}) = 0$.

\noindent Then, $c(f; s_{120} s_{111}^{2} s_{102}) = c(f; s_{201} s_{111}^{2} s_{021}) = -4$.

For $\epsilon = (2, 2, 2)$, $c(f; s_{201} s_{111}^{2} s_{021}) + \frac{1}{2}c(f; s_{111}^{4}) + c(f; s_{120} s_{111}^{2} s_{102}) + c(f; s_{210} s_{111}^{2} s_{012}) = 0$ in the end. Then, $c(f; s_{111}^{4})$ can be determined as $24$ due to the previous relations.
\end{enumerate}

In this case, all 25 coefficients of $f\in K_{(4,4,4)}$ are decided if we set $c(f; s_{210}^{2} s_{021} s_{003}) = -2$ (i.e. $f$ is determined up to scalar multiplications) and thus $K_{(4,4,4)} = \CC f$. Actually, this $f$ is equal to the Aronhold invariant $A_f$ in (\ref{Aronhold}); for example, $c(f; s_{111}^{4}) = 24$ means that the coefficient of $f$ with respect to $\frac{1}{4!} s_{111}^{4}$ is $24$, that is, the coefficient of $f$ with respect to $s_{111}^{4}$ is $1$. We can determine the defining polynomial of $\sigma_3(v_3(\PP^2))$ in this way.

This process is a \emph{finite} computation for given $k,d,n$, which means it can always be done by hand in principle and can be implemented easily. We develop this idea to provide a combinatorial procedure for computing the dimension of $K_{\beta}$ in \textsection{}\ref{sec:some-small-degrees}.
\end{example}

Next we discuss the case when $\beta = (b_0, b_1, \dots, b_n)$ has symmetry,
e.g., $b_{i_1} = b_{i_2}$ for some $i_1 \neq i_2$,
which we will use in \textsection{}\ref{sec:some-small-degrees}.
In this case, since $\phi_{(i_1 i_2)}(K_{\beta}) = K_{\beta}$ for the permutation $(i_1 i_2)$,
it can be further decomposed as the sum of smaller subspaces.

In general, let $\sA \subset \Aut(\CC[\set{s_\alpha}]_{k+1})$
be any finite subgroup of the linear automorphism group and $W \subset \CC[\set{s_\alpha}]_{k+1}$ be a subspace satisfying $\phi(W) = W$ for each $\phi \in \sA$; we call such a subspace an \emph{$\sA$-invariant subspace} of $\CC[\set{s_\alpha}]_{k+1}$.
Then, for any $\sA$-invariant subspace $W$, we take
$\lambda_{\sA}: W \rightarrow W$ to be the averaging linear operator $f \mapsto \lambda_{\sA}(f) := \frac{1}{\#\sA}\sum_{\phi \in \sA} \phi(f)$. Since $\phi(\lambda_{\sA}(f)) = \lambda_{\sA}(f)$ for any $f\in W$ and $\phi\in\sA$, it follows that
$\lambda_{\sA}(\lambda_{\sA}(f)) = \lambda_{\sA}(f)$, 
$\lambda_{\sA}^2 = \lambda_{\sA}$, and $W \simeq \im(\lambda_{\sA}) \oplus \ker(\lambda_{\sA})$.

\begin{lemma}\label{lem:decomp-K}
  Let $\sA, W$ be as above.
  Let $K \subset W$ be also an $\sA$-invariant subspace.
  Then,
  $K \cap \im(\lambda_{\sA}) = \im(\lambda_{\sA}|_K)$
  and
  $K \cap \ker(\lambda_{\sA}) = \ker(\lambda_{\sA}|_K)$.
  Hence
  \begin{gather*}
    K \simeq (K \cap \im(\lambda_{\sA})) \oplus (K \cap \ker(\lambda_{\sA})).
  \end{gather*}
\end{lemma}
\begin{proof}
  Since $\phi(K) = K$ for all $\phi \in \sA$,
  we have $\lambda_{\sA}(K) \subset K$.
  It induces a linear map $\lambda_{\sA}|_K: K \rightarrow K$.
  Then, $K \cap \ker(\lambda_{\sA}) = \ker(\lambda_{\sA}|_K)$. Also, it is immediate to see that 
  $\im(\lambda_{\sA}|_K) \subset K \cap \im(\lambda_{\sA})$.

  On the other hand, for any $g \in K \cap \im(\lambda_{\sA})$ (i.e., $g=\lambda_{\sA}(f) \in K$ for some $f\in W$) we have $g=\lambda_{\sA}(f) = \lambda_{\sA}(\lambda_{\sA}(f))=\lambda_{\sA}(g)$, which means $g\in \im(\lambda_{\sA}|_K)$.
  Hence $\im(\lambda_{\sA}|_K) = K \cap \im(\lambda_{\sA})$.
\end{proof}

\begin{remark}[The case when $K = \ker(\Psi|_W)$]\label{rem:W-sA-basis}
  Let $\sA, W$ be as above.
  Take $K = \ker(\Psi|_W)$ and assume that $K$ is also an $\sA$-invariant subspace.
  To compute {$K \cap \ker(\lambda_{\sA}) = \ker(\Psi|_{\ker(\lambda_{\sA})})$ and
  $K \cap \im(\lambda_{\sA}) = \ker(\Psi|_{\im(\lambda_{\sA})})$},
  we may need to know bases of
  subspaces $\im(\lambda_{\sA})$ and $\ker(\lambda_{\sA})$.
  \medskip

  \begin{inparaenum}[(a)]
  \item \label{rem:W-sA-basis-a}
    We write a finite set $\sA f = \set{\phi(f) \mid \phi \in \sA}$ for any $f \in W$
    (note that $\#\sA f < \# \sA$ can occur,
    for example, in the case when $\phi(f) = f$ for some $\phi\in\sA$). We say that $r$ polynomials \textit{$f_1, \dots, f_r$ give an $\sA$-basis of $W$}
    if $\sA f_i \cap \sA f_j = \emptyset$
    for $i \neq j$ and the union $\bigcup_{i=1}^r \sA f_i$ gives a basis of $W$.
    In this case, the linear span $\lin{\bigcup_{i=1}^r \sA f_i} = \bigoplus_{i=1}^r \lin{\sA f_i}$,
    and
    $\lambda_{\sA}(f_1), \dots, \lambda_{\sA}(f_r)$ give a basis of $\im(\lambda_{\sA})$.
    \medskip

  \item
    In order to study $\ker(\lambda_{\sA})$,
    we set $a = \#\sA$, $\sA = \set{\phi_1, \dots, \phi_{a}}$, and take
    a $a \times (a -1)$ matrix $M = [M_1| \dots |M_{a-1}]$ over $\CC$ of rank $a-1$ such that
    the $j$-th column vector $M_j = [\{M_{l,j}\}_{l=1}^{a}]$ satisfies
    $\sum_{l=1}^{a} M_{l,j} = 0$ for each $1 \leq j \leq a-1$.
    Setting $\sA f M_j := [\phi_1(f)| \dots |\phi_a(f)] \cdot M_j=M_{1,j}\phi_1(f)+\cdots+M_{a,j}\phi_a(f)$, we have
    $\sA f M_j \in \ker \lambda_{\sA}$. On the other hands,
    \begin{gather*}
      [\phi_1(f)| \dots |\phi_a(f)] \cdot M = [\sA f M_1| \dots| \sA f M_{a-1}].
    \end{gather*}
    We take $\sA f^M$ to be a basis of $\lin{A f M_1, \dots, \sA f M_{a-1}}$,
    which is a subspace of $\lin{\sA f}$.
    For an $\sA$-basis $f_1, \dots, f_r$ of $W$,
    the union $\bigcup_{i=1}^r \sA f_i^M$ gives a basis of $\ker(\lambda_{\sA})$.
    The reason is as follows.

    It is linearly independent, because of
    $\bigoplus \lin{\sA f_i^M} \subset \bigoplus \lin{\sA f_i}$. For ``spanning'', since $f_1, \dots, f_r$ give an $\sA$-basis of $W$, any $f \in W$ is written by
    a sum of members of $\bigcup_{i=1}^r \sA f_i$ over $\CC$.
    Let $C(f; \phi(f_i)) \in \CC$ be the coefficient of $f$
    with respect to $\phi(f_i)$.
    If $f \in \ker(\lambda_{\sA})$, then we have $\sum_{l=1}^a C(f; \phi_l(f_i)) = 0$ for each $1 \leq i \leq r$. Taking $a \times a$ matrix $N_i = [M ~|~ C_i]$ with a column vector $C_i = [\set{C(f; \phi_l(f_i))}_{l=1}^a]$,
    since the sum of row vectors of $N_i$ is equal to the zero row vector,
    we have $\rank N_i = a-1$; in particular, $C_i$
    is a linear combination of column vectors $M_1, \dots, M_{a-1}$.
    Hence
    $\sA f_i C_i = \sum_{l=1}^a C(f; \phi_l(f_i)) \cdot \phi_l(f_i)$
    is written as a sum of $\sA f_i M_1, \dots, \sA f_i M_{a-1}$.
    It means that $f$ is a sum of members of $\bigcup_{i=1}^r \sA f_i^M$.
    \medskip

  \item
    We also have a chance to reduce the computation of $\Psi(f) = \sum \Psi_m(f) e(m)$
    using a finite subgroup of the permutation group.
    Recall that $\Psi_m(f)$ be as in the formula~(\ref{Psi-m-f}) for a monomial $m$ of degree $k-1$.
    For a permutation $\tau$ and the linear automorphism $\phi_\tau$,
    $\Psi_{\phi_{\tau}(m)}(\phi_{\tau}(f))$ is equal to the image of $\Psi_{m}(f)$
    under the automorphism of $\CC[t_0, t_1, \dots, t_n]$ induced by $\tau$.
    Let $G \subset \mathfrak{S}_{n+1}$ be a subgroup and $f$ satisfy $\phi_{\tau}(f) \in \CC f$ for any $\tau \in G$.
    Then, for any $\tau \in G$ we have $\Psi_{\phi_{\tau}(m)}f = 0$ if and only if $\Psi_m(f) = 0$.

    Furthermore, let $\sA = \set{\phi_{\tau} \mid \tau \in G}$,
    and take monomials $m_1, \dots, m_s$ giving $\sA$-basis of $\CC[\set{s_\alpha}]_{k-1}$.
    If a subspace $W' \subset \CC[\set{s_\alpha}]_{k+1}$ satisfies
    $\phi_{\tau}(f) \in \CC f$ for any $f \in W'$ and every $\tau\in G$ (e.g., $W' = \im(\lambda_{\sA})$),
    then instead of computing $\Psi_m|_{W'}$ for all $m\in \CC[\set{s_\alpha}]_{k-1}$,
    it is sufficient to compute it for $m_1, \dots, m_s$.
  \end{inparaenum}
\end{remark}

\begin{remark}\label{split-inv-ker-im}
  \begin{inparaenum}[(a)]
  \item \label{im-ker-sigma}
    Let $\beta \in \ZZ_{\geq 0}^{n+1}$ satisfy
    $b_{i_1} = b_{i_2}$ for some $i_1 \neq i_2$,
    and consider the permutation $\tau = (i_1 i_2)$
    (e.g., $\beta=(10,9,9,8)$ and $\tau = (12)$).
    Taking the group $\lin{\phi_{\tau}} = \set{\id, \phi_{\tau}}$ of order $2$,
    we have a linear map
    $\lambda_{\tau} = \lambda_{\lin{\phi_{\tau}}}: L_{\beta} \rightarrow L_{\beta}$
    defined by $\lambda_{\tau}(f) = \frac{1}{2}(f + \phi_{\tau}(f))$,
    and $L_{\beta} \simeq \im(\lambda_{\tau}) \oplus \ker(\lambda_{\tau})$.

    The vector space $\im(\lambda_{\tau})$ is generated by
    $m+\phi_{\tau}(m)$ with monomials $m \in M(L_{\beta})$.
    On the other hand, if $f \in L_{\beta}$ is contained in $\ker(\lambda_{\tau})$, then
    the coefficients $c(f; m)$ and $c(f; \phi_{\tau}(m))$ satisfy
    $c(f; m)+c(f; \phi_{\tau}(m)) = 0$.
    It means that $\ker(\lambda_{\tau})$ is generated by
    $m-\phi_{\tau}(m)$ with $m \in M(L_{\beta})$.
    In this way, we may choose bases of vector spaces
    $\im(\lambda_{\tau})$ and $\ker(\lambda_{\tau})$
    as in (a) and (b) of Remark~\ref{rem:W-sA-basis}.
    If a monomial $m$ satisfies $\phi_{\tau}(m) = m$,
    then $m \in \im(\lambda_{\tau})$ and $m \notin \ker(\lambda_{\tau})$.

    We write $L_{\beta}[\tau] := \im(\lambda_{\tau})$ and $L_{\beta}[-\tau] := \ker(\lambda_{\tau})$.
    For the kernel $K_{\beta} = \ker(\Psi|_{L_{\beta}})$,
    setting $K_{\beta}[\pm\tau]$ by
    $K_{\beta}[+\tau] = K_{\beta}[\tau] := K_{\beta} \cap (L_{\beta}[\tau])$ and
    $K_{\beta}[-\tau] := K_{\beta} \cap (L_{\beta}[-\tau])$,
    and using Lemma~\ref{lem:decomp-K},
    we have
    \begin{gather*}
      K_{\beta} \simeq K_{\beta}[\tau] \oplus K_{\beta}[-\tau].
    \end{gather*}

  \item
    In the setting of (\ref{im-ker-sigma}), in addition we assume that $\beta \in \ZZ_{\geq 0}^{n+1}$ satisfies
    $b_{j_1} = b_{j_2}$ for some $j_1 \neq j_2$ with
    $j_1, j_2 \notin \{i_1, i_2\}$,
    and take $\tau' = (j_1 j_2)$; then $\tau$ and $\tau'$ are disjoint
    (e.g., $\beta=(10,10,8,8)$, $\tau = (01)$, and $\tau' = (23)$).
    It follows that $\phi_{\tau'}(L_{\beta}[\pm \tau]) = L_{\beta}[\pm \tau]$.
    We write $L_{\beta}[\pm \tau, \tau'] := \lambda_{\tau'}(L_{\beta}[\pm \tau])$
    and
    $L_{\beta}[\pm \tau, -\tau'] = \ker(\lambda_{\tau'}) \cap (L_{\beta}[\pm \tau])$.
    The vector spaces
    $L_{\beta}[\tau, \pm\tau']$, $L_{\beta}[-\tau, \pm\tau']$
    are generated by
    \begin{gather*}
      m+\phi_{\tau}(m) \pm \phi_{\tau'}(m) \pm \phi_{\tau'\tau}(m),\;
      m-\phi_{\tau}(m) \pm \phi_{\tau'}(m) \mp\phi_{\tau'\tau}(m)
    \end{gather*}
    with 
    $m \in M(L_{\beta})$,
    respectively.
    Setting $K_{\beta}[\pm\tau, \pm\tau'] := K_{\beta} \cap (L_{\beta}[\pm\tau, \pm\tau'])$, we have a decomposition
    \begin{gather*}
      K_{\beta} \simeq K_{\beta}[\tau, \tau'] \oplus K_{\beta}[\tau, -\tau'] \oplus K_{\beta}[-\tau, \tau'] \oplus K_{\beta}[-\tau, -\tau'].
    \end{gather*}

  \end{inparaenum}
\end{remark}

\section{Equations and Degree of $\sigma_4(v_3(\PP^3)) \subset \PP^{19}$ }\label{sect_s4(v3(P3))}

\subsection{36 quintic polynomials}\label{sub_sect_FGH}

In this section, we study the $4$-secant variety $\sigma_4(v_3(\PP^3))$ of the Veronese embedding
$v_3: \PP^3 \rightarrow \PP^{19}$ and
the ring homomorphism
$v_3^*: S(\PP^{19}) = \CC[\set{s_\alpha}] \rightarrow S(\PP^3)=\CC[t_0,t_1,t_2,t_3]$.
From Proposition~\ref{prop-dsum-beta}, in order to find generators of
$I(\sigma_4(v_3(\PP^3)))_{5}$,
it is sufficient to calculate the kernel $K_{\beta} = \ker(\Psi|_{L_{\beta}})$ for each direct summand $L_{\beta}$ of $\CC[\set{s_\alpha}]_5$
with $\beta \in \ZZ_{\geq 0}^4$ and $|\beta| = 15$.
We first present $3$ polynomials $F,G,H$ of $\CC[\set{s_\alpha}]_5$, each of which is contained in $K_{\beta}$
with $\beta=(2,4,4,5),(3,4,4,4),(3,3,4,5)$ respectively (Definitions~\ref{def-FH-1} and \ref{def-G}), and next
give a minimal set of polynomials spanning
$I(\sigma_4(v_3(\PP^3)))_{5}$ in Theorem~\ref{FGH-prolong}.

\begin{definition}\label{def-FGH-mat-A1}
  We set $A_1$ to be the following $11 \times 11$ skew-symmetric matrix,
  where $2445, \dots, 3435$ are \emph{labels of columns},
  and $4554, \dots, 3564$ are \emph{labels of rows} (These labels are not entries of the matrix).
  \begin{gather*}
    \begin{matrix}
      \ &2445&3534&3543&3453&3444(1)&3354&2544&2454&3345&3444(2)&3435
      \\
      4554&0&s_{1020}&s_{1011}&s_{1101}&s_{1110}&s_{1200}&s_{2010}&s_{2100}&0&0&0
      \\
      3465&-s_{1020}&0&0&s_{0012}&0&s_{0111}&0&s_{1011}&s_{0120}&s_{0021}&s_{0030}
      \\
      3456&-s_{1011}&0&0&s_{0003}&0&s_{0102}&0&s_{1002}&s_{0111}&s_{0012}&s_{0021}
      \\
      3546&-s_{1101}&-s_{0012}&-s_{0003}&0&-s_{0102}&0&-s_{1002}&0&s_{0201}&s_{0102}&s_{0111}
      \\
      3555&-s_{1110}&0&0&s_{0102}&0&s_{0201}&0&s_{1101}&s_{0210}&s_{0111}&s_{0120}
      \\
      3645&-s_{1200}&-s_{0111}&-s_{0102}&0&-s_{0201}&0&-s_{1101}&0&s_{0300}&s_{0201}&s_{0210}
      \\
      4455&-s_{2010}&0&0&s_{1002}&0&s_{1101}&0&s_{2001}&s_{1110}&s_{1011}&s_{1020}
      \\
      4545&-s_{2100}&-s_{1011}&-s_{1002}&0&-s_{1101}&0&-s_{2001}&0&s_{1200}&s_{1101}&s_{1110}
      \\
      3654&0&-s_{0120}&-s_{0111}&-s_{0201}&-s_{0210}&-s_{0300}&-s_{1110}&-s_{1200}&0&0&0
      \\
      3555&0&-s_{0021}&-s_{0012}&-s_{0102}&-s_{0111}&-s_{0201}&-s_{1011}&-s_{1101}&0&0&0
      \\
      3564&0&-s_{0030}&-s_{0021}&-s_{0111}&-s_{0120}&-s_{0210}&-s_{1020}&-s_{1110}&0&0&0
    \end{matrix}
  \end{gather*}
  
  Since there are two columns having $3444$,
  to distinguish them we add $(1), (2)$.
  Note that the subscript of the monomial of each entry of $A_1$ is equal
  to the label of the row minus the label of column (e.g., for the $(1,2)$-entry $s_{1020}$,
  we have $(1,0,2,0) = (4,5,5,4) - (3,5,3,4)$).
  We may calculate weights of minors of $A_1$ over $\CC[t_0,t_1,t_2,t_3]$
  by using these labels.
  Also, note that the sum of labels of columns is
  $(2,4,4,5) + \dots + (3,4,3,5) = (30, 45, 45, 45)$, whereas the sum of labels of rows is
  $(4,5,5,4) + \dots + (3,5,6,4) = (36, 54, 54, 54)$.
\end{definition}

\begin{lemma}\label{lem:A_1-rk2}
  We have $\rank(A_1(w)) = 2$ for any $w \in v_3(\PP^3)$.
  Hence, for any $p \times p$ submatrix $L$ of $A_1$ with $p > 2k$,
  it holds $\det(L) \in I(\sigma_k(v_3(\PP^3)))$.
\end{lemma}

\begin{proof}
  To check $\rank(A_1(w)) = 2$ for any $w \in v_3(\PP^3)$,
  we consider the matrix $v_d^*(A_1)$ whose entries are monomials of $\CC[t_0,t_1,t_2,t_3]$ by evaluation.
  Setting $v_{l}$ to be the row of $v_d^*(A_1)$ with label $l$,
  we may check each row vector is written as a sum of two rows $v_{4554}$ and $v_{3465}$.
  For example, we have $v_{3654} = - t^{(-1,1,0,0)} v_{4554}$
  with $t^{(-1,1,0,0)} = {t_1}/{t_0}$.
  Similarly, $v_{3546} =  t^{(0,1,-2,1)} v_{3465} - t^{(-1,0,-1,2)} v_{4554}$.

  Next,
  let $p,k$ be numbers with $p > 2k$, and
  let $L$ be a $p \times p$ submatrix of $A_1$.
  For a general $w \in \sigma_k(v_3(\PP^3))$, one can write as $w = \sum_{i=1}^k w_i$ with $w_i \in v_3(\PP^3)$.
  Then $L(w) = \sum_{i=1}^k L(w_i)$ and by subadditivity of matrix rank,
  it follows $\rank(L(w)) \leq k \times 2 < p$.
  Then $\det(L) \in I(\sigma_k(v_3(\PP^3)))$.
\end{proof}

\begin{definition}\label{def-FH-1}
  We define $\xi \in I(\sigma_3(v_3(\PP^3)))_4$ and
  $F, H \in I(\sigma_4(v_3(\PP^3)))_5$ as follows.
  Let $F$ be the Pfaffian of $10\times 10$ skew-symmetric submatrix
  given by removing the 2445 column and the 4554 row from $A_1$.
  Then $F$ is of degree $5$.
  Moreover, $F^2$, the determinant of the matrix,
  is of weight
  $(36, 54, 54, 54) - (30, 45, 45, 45) - (4,5,5,4) + (2,4,4,5) = (4,8,8,10)$ over $\CC[t_0,t_1,t_2,t_3]$.
  Hence $F$ is of weight $(2,4,4,5)$. By Lemma~\ref{lem:A_1-rk2},
  $F^2 \in I(\sigma_4(v_3(\PP^3)))$ and hence
  $F \in I(\sigma_4(v_3(\PP^3)))$.
  Note that $F$ does not have variables $s_{2010}, s_{2100}, s_{3000}$.

  Let $\xi$ be the Pfaffian of $8\times 8$ skew-symmetric submatrix
  given by removing 2445, 2544, 2454 columns and 4554, 4455, 4545 rows from $A_1$.
  Then $\xi$ is of degree $4$ and of weight $(0,4,4,4)$.
  Using a recursive expansion of Pfaffians, we have
  \begin{gather*}
    F = s_{2001} \cdot \xi  + F'
  \end{gather*}
  where $F'$ does not have the variable $s_{2001}$.
  
  Let $H$ be the Pfaffian of $10\times 10$ skew-symmetric submatrix
  given by removing 3345 column and 3654 row from $A_1$,
  which is of degree $5$ and of weight $(3,3,4,5)$.
\end{definition}

\begin{definition}\label{def-G}
  Let $G^a, G^b \in \CC[\set{s_\alpha}]$ be two quintic polynomials as below,
where 
$\phi_{(12)}(G^a) = G^a$, $\phi_{(12)}(G^b) \neq G^b$.
We set $G \in \CC[\set{s_\alpha}]$ of degree $5$ by
\begin{gather*}
  G = s_{3000}\cdot \xi+G^a+G^b+\phi_{(12)}(G^b),
\end{gather*}
where it holds $\phi_{(12)}(G) = G$ and
the weight of $G$ is $(3,4,4,4)$. Also, one can check $\Psi|_{L_{(3,4,4,4)}}(G) = 0$ by a direct computation;
hence $G \in K_{(3,4,4,4)}$.
Note that in \textsection{}\ref{sec:minim-free-resol}, we will also give another expression of $G$ in terms of determinants of some matrices as well as $F$ and $H$ (see Proposition~\ref{prop:G-from-mat}).
{
\footnotesize

\begin{align*}
G^a =& -(s_{2001} s_{1110} s_{0003} - s_{2001} s_{1002} s_{0111} - s_{1110} s_{1002}^{2} + s_{1101} s_{1011} s_{1002}) (s_{0300} s_{0030} - s_{0210} s_{0120})\\
    &-2 (s_{1200} s_{1110} s_{1020} - s_{1110}^{3}) (s_{0111} s_{0003} - s_{0102} s_{0012})
    -(s_{1200} s_{1020} s_{1002} - 2 s_{1110} s_{1101} s_{1011}) (s_{0201} s_{0021} + s_{0111}^{2})\\
    &-s_{1110}^{2} s_{1002} (s_{0201} s_{0021} + 3 s_{0111}^{2})
\end{align*}
}
{
\footnotesize
\begin{align*}
G^b =&-(s_{2100} s_{1200} s_{0120} - 2 s_{2100} s_{1110} s_{0210} + s_{2100} s_{1020} s_{0300} - s_{1200}^{2} s_{1020} + s_{1200} s_{1110}^{2}) (s_{0021} s_{0003} - s_{0012}^{2})\\
&+(s_{2100} s_{1200} s_{0111} - 2 s_{2100} s_{1101} s_{0210} + s_{2001} s_{1200} s_{0210} - s_{1200}^{2} s_{1011} + s_{1200} s_{1110} s_{1101}) (s_{0030} s_{0003} - s_{0021} s_{0012})\\
&-(s_{2100} s_{1200} s_{0102} - 2 s_{2100} s_{1101} s_{0201} + s_{2100} s_{1002} s_{0300} - s_{1200}^{2} s_{1002} + s_{1200} s_{1101}^{2}) (s_{0030} s_{0012} - s_{0021}^{2})\\
&+(2 s_{2100} s_{1002} s_{0102} - s_{1200} s_{1002}^{2} + s_{1101}^{2} s_{1002}) (s_{0210} s_{0030} - s_{0120}^{2})\\
&-s_{2100} \bigl\{
 2 s_{1110} (s_{0120} s_{0111} s_{0003} - s_{0120} s_{0102} s_{0012} - s_{0111}^{2} s_{0012} + s_{0111} s_{0102} s_{0021})\\
& \quad \quad\quad  -2 s_{1101} (s_{0120}^{2} s_{0003} - 2 s_{0120} s_{0111} s_{0012} + s_{0111}^{2} s_{0021})\\
& \quad \quad\quad  -s_{1020} (s_{0210} s_{0111} s_{0003} - s_{0210} s_{0102} s_{0012} - s_{0201} s_{0111} s_{0012} + s_{0201} s_{0102} s_{0021})\\
& \quad \quad\quad -s_{1011} (s_{0300} s_{0030} s_{0003} - s_{0300} s_{0021} s_{0012} - s_{0210} s_{0120} s_{0003} + 2 s_{0210} s_{0111} s_{0012} - s_{0210} s_{0102} s_{0021} \\ 
& \quad \quad\quad\quad \quad\quad - s_{0201} s_{0120} s_{0012} + 2 s_{0201} s_{0111} s_{0021} - s_{0201} s_{0102} s_{0030} + 2 s_{0120} s_{0111} s_{0102} - 2 s_{0111}^{3})\\
& \quad \quad\quad -s_{1002} (s_{0210} s_{0120} s_{0012} - 3 s_{0210} s_{0111} s_{0021} + s_{0201} s_{0120} s_{0021} - s_{0201} s_{0111} s_{0030} + 2 s_{0120} s_{0111}^{2})
\bigr\}\\
&-s_{2001} \bigl\{
 s_{1200} (s_{0120}^{2} s_{0003} - s_{0120} s_{0111} s_{0012} - s_{0120} s_{0102} s_{0021} + s_{0111} s_{0102} s_{0030})-s_{1110} (s_{0300} s_{0021} s_{0012} - s_{0210} s_{0102} s_{0021})\\
& \quad \quad\quad +s_{1002} (s_{0300} s_{0120} s_{0021} - s_{0210}^{2} s_{0021})
\bigr\}\\
&+(s_{1200} s_{1011} s_{1002} - s_{1110} s_{1101} s_{1002}) (s_{0210} s_{0021} + s_{0201} s_{0030} - 2 s_{0120} s_{0111})\\
&+(3 s_{1200} s_{1110} s_{1011} - s_{1200} s_{1101} s_{1020} - 2 s_{1110}^{2} s_{1101}) s_{0120} s_{0003}\\
&-(s_{1200} s_{1110} s_{1011} + s_{1200} s_{1101} s_{1020} - 2 s_{1110}^{2} s_{1101}) s_{0102} s_{0021}
-2 (s_{1200} s_{1110} s_{1011} - s_{1200} s_{1101} s_{1020}) s_{0111} s_{0012}\\
&+2 (2 s_{1200} s_{1110} s_{1002} - s_{1200} s_{1101} s_{1011} - s_{1110} s_{1101}^{2}) s_{0111} s_{0021}
-(s_{1200} s_{1110} s_{1002} - s_{1200} s_{1101} s_{1011}) s_{0102} s_{0030}\\
&-(3 s_{1200} s_{1110} s_{1002} - s_{1200} s_{1101} s_{1011} - 2 s_{1110} s_{1101}^{2}) s_{0120} s_{0012}+(s_{1200} s_{1020} s_{1002} + 2 s_{1110}^{2} s_{1002} - 2 s_{1110} s_{1101} s_{1011}) s_{0210} s_{0012}\\
&
+s_{1200} s_{1011}^{2} s_{0111}^{2}
-s_{1200} s_{1011}^{2} s_{0120} s_{0102}
\end{align*}

}
\end{definition}

Furthermore, we note that one can have $11$ Pfaffians of degree $5$ whose weights correspond to labels of all columns of $A_1$ in a similar manner.

\begin{remark}
  (a)
  Let $\mathfrak{S}_{4}$ be the full symmetric group on $\set{0,1,2,3}$, and let
  $\phi_{\tau}$ be the automorphism on $\CC[\set{s_\alpha}]$
  for any $\tau \in \mathfrak{S}_{4}$ as in (\ref{eq:aut-phi-sigma}). Then we have $\phi_{(12)}(F) = F$ for $(12) \in \mathfrak{S}_4$ by computing directly or by the following discussion.
  Since $F \in I(\sigma_4(v_3(\PP^3)))$,
  we have $F \in K_{(2,4,4,5)} = \ker(\Psi|_{L_{(2,4,4,5)}})$
  in the setting of Proposition~\ref{prop-dsum-beta}.
  Later in Theorem~\ref{FGH-prolong} and Appendix~\ref{rem:coeff-36-poly},
  we in fact show $K_{(2,4,4,5)} = \CC F$.
  Since $\phi_{(12)}(F) \in K_{(2,4,4,5)}$,
  it holds $\phi_{(12)}(F) = cF$ with some $c \in \CC$.
  Comparing coefficients of some terms of $F$ and $\phi_{(12)}(F)$,
  we may check $\phi_{(12)}(F) = F$.

  Similarly,
  we have $\phi_{(12)}(H) = H$.

  (b) Since $\xi \in I(\sigma_3(v_3(\PP^3)))$,
  we have $\xi \in K^{(3)}_{(0,4,4,4)} = \ker(\Psi^{(3)}|_{L_{(0,4,4,4)}})$.
  In fact, we have that $K^{(3)}_{(0,4,4,4)} = \CC \xi$, and
  $\xi$ corresponds to the \emph{Aronhold invariant}, which is studied in Example~\ref{eg_Aronhold},
  under the identification $s_{0ijk} \mapsto s_{ijk}$;
  this is because, the same computation in Example~\ref{eg_Aronhold} implies that $K^{(3)}_{(0,4,4,4)}$ is of dimension $1$, and
  coefficients of its generator are equal to those of the Aronhold invariant up to scalars.
  In particular, we have $\phi_{(12)}(\xi) = \phi_{(13)}(\xi) = \phi_{(23)}(\xi) = \xi$ as the above discussion for $F$.
\end{remark}

\begin{remark}
  As in Theorem~\ref{FGH-prolong}, $\dim K_{(3,4,4,4)} = 3$,
  which is not similar to the case of $\beta = (2,4,4,5), (3,3,4,5)$.
From the matrix $A_1$ in Definition~\ref{def-FGH-mat-A1}, it is possible to have $2$ Pfaffians belonging to $K_{(3,4,4,4)}$ by removing one of two 3444 columns and a corresponding row;
  however, these Pfaffians do not have variable $s_{3000}$, and are not enough
  to generate $G$ and also the whole space $K_{(3,4,4,4)}$.
\end{remark}

Let us consider the open subset $D(\xi) = \PP^{19} \setminus V(\xi)$,
and take $4$ homogeneous polynomials
$F_1 = F$, $F_2=\phi_{(23)}(F)$, $F_3=\phi_{(13)}(F)$, $G$
of weight
$(2, 4, 4, 5), (2, 4, 5, 4), (2, 5, 4, 4), (3,4,4,4)$, respectively.
These polynomials are expressed as follows:
\begin{gather*}
  F_1 = s_{2001}\cdot \xi+F_1',\ \,
  F_2 = s_{2010}\cdot \xi+F_2',\ \,
  F_3 = s_{2100}\cdot \xi+F_3',\ \,
  G = s_{3000}\cdot \xi+G',
\end{gather*}
where $\xi, F_1', F_2', F_3'$
does not have variables $s_{2001}, s_{2010}, s_{2100}, s_{3000}$, and
$G'$
does not have the variable $s_{3000}$.
Then, in the localization ring $(\CC[\set{s_\alpha}]/(F_1, F_2, F_3, G))_{\xi}$,
we have
\begin{gather}\label{eq:qring_iso_local__s}
  s_{2001} = -F_1'/\xi,\ \,
  s_{2010} = -F_2'/\xi,\ \,
  s_{2100} = -F_3'/\xi,\ \,
  s_{3000} = -G'/\xi.
\end{gather}

Let $R = \CC[\set{s_\alpha \mid s_\alpha \neq s_{2001}, s_{2010}, s_{2100}, s_{3000}}]$,
a smaller polynomial ring of $16$ variables.
It implies a ring isomorphism
\begin{gather}\label{eq:qring_iso_local}
  \rho: (\CC[\set{s_\alpha}]/(F_1, F_2, F_3, G))_{\xi} \simeq R_{\xi},
\end{gather}
where $\xi$ is regarded as a polynomial of $R$ in the right hand side.

\begin{lemma}\label{iso-outside-xi}
  We have
  $\sigma_4(v_3(\PP^3)) \cap D(\xi) = V(F_1, F_2, F_3, G) \cap D(\xi)$.
  In particular, $\sigma_4(v_3(\PP^3)) \cap D(\xi)$
  is a smooth variety.
\end{lemma}

\begin{proof}
  Since $\sigma_4(v_3(\PP^3)) \subset V(F_1, F_2, F_3, G)$, we have
  $\sigma_4(v_3(\PP^3)) \cap D(\xi) \subset V(F_1, F_2, F_3, G) \cap D(\xi)$.
  Let us consider the coordinate ring of $V(F_1, F_2, F_3, G) \cap D(\xi)$,
  which is given as the ring of elements of degree $0$ of the left hand side of
  the isomorphism~(\ref{eq:qring_iso_local}).
  Since it is integral domain, $V(F_1, F_2, F_3, G) \cap D(\xi)$
  is an irreducible and reduced subset of dimension $15$.
  Since $\dim(\sigma_4(v_3(\PP^3))) = 15$,
  the equality of the statement holds.
  Moreover, regarding $\xi$ as an polynomial of $R$ and taking
  $V'(\xi) \subset \PP^{15}$ to be the zero set of $\xi \in R$,
  we have $V(F_1, F_2, F_3, G) \cap D(\xi) \simeq \PP^{15} \setminus V'(\xi)$,
  which is a smooth variety.
\end{proof}

\begin{proof}[Proof of Corollary \ref{s4v3P3-smoothness}]
  There is a normal form classification for ternary cubic forms (see e.g., \cite[Table 10.4.1]{La}).
  Thus, to check non-singularity of any point $w$ in $\sigma_{4}(v_3(\PP^2))\setminus\sigma_{3}(v_3(\PP^3))$, we may assume that $w$ is one of the following forms
  i) $x_1^2 x_2-x_0^3-x_0 x_2^2$, ii) $x_0 x_1 x_2$,
  iii) $x_0(x_0^2 +x_1 x_2)$, iv) $x_1^2 x_2-x_0^3-a x_0 x_2^2-b x_2^3$ ($a,b\in\CC$: general).

  Recall that $\xi$ is the Pfaffian of a $8 \times 8$ skew symmetric matrix
  as in Definition~\ref{def-FH-1}.
  For each form $w$ in i)-iv) above,
  one can directly check $\xi(w) \neq 0$ (indeed, $\xi(w)^2 = 1, 1, 1, a^2$, respectively)
  by calculating the matrix.
  For example, for iv),
  $w$ corresponds to a point of $\PP^{19}$
  defined by $s_{0021} =1$, $s_{0300} =-1$, $s_{0102} =-a$, $s_{0003} =-b$ and other $s_\alpha=0$
  in our notation, and then the determinant of the matrix $\xi(w)^2 = a^2$.

  Hence, $w \in \sigma_4(v_3(\PP^3)) \cap D(\xi)$
  and is a non-singular point of $\sigma_4(v_3(\PP^3))$
  because of Lemma~\ref{iso-outside-xi}.
\end{proof}

\begin{theorem}\label{FGH-prolong}
  Let $K_{\beta} = \ker \Psi|_{L_{\beta}} \subset L_{\beta}$
  with $\beta \in \ZZ_{\geq 0}^4$ and $|\beta| = 15$,
  as in Proposition~\ref{prop-dsum-beta}.
  Then we have
  \begin{gather}\label{eq:FGH-prolong}
    K_{(2,4,4,5)} = \CC F,\
    K_{(3,4,4,4)} = \CC \langle G,\,\phi_{(23)}(G),\,\phi_{(13)}(G)\rangle,\
    K_{(3,3,4,5)} = \CC H,
  \end{gather}
  and $K_{\beta} = 0$
  for any $\beta \notin {\mathfrak{S}_{4}\cdot(2,4,4,5) \cup \mathfrak{S}_{4}\cdot(3,4,4,4) \cup \mathfrak{S}_{4}\cdot(3,3,4,5)}$.
  Hence it holds the decomposition of nonzero direct summands
  \begin{gather}\label{eq:FGH-prolong-2}
    I(\sigma_4(v_3(\PP^3)))_{5}
    \;= \bigoplus_{\beta \in {\mathfrak{S}_{4}\cdot(2,4,4,5) \cup \mathfrak{S}_{4}\cdot(3,4,4,4) \cup \mathfrak{S}_{4}\cdot(3,3,4,5)}} K_{\beta},
  \end{gather}
  and $\dim I(\sigma_4(v_3(\PP^3)))_{5} = 36$.
  In particular, a basis consisting of $36$ polynomials of $I(\sigma_4(v_3(\PP^3)))_5$
  is exactly obtained by $\phi_{\tau}(P)$ with $P = F,G,H$ and suitable $\tau \in \mathfrak{S}_4$.
\end{theorem}

\begin{proof}
  Since $F,G,H \in I(\sigma_4(v_3(\PP^3)))$,
  we have $F \in K_{(2,4,4,5)}$,
  $G \in K_{(3,4,4,4)}$, and
  $H \in K_{(3,3,4,5)}$.
  By symmetry on $v_4(\PP^3) \subset \PP^{19}$,
  it holds that $\phi_{(23)}(G), \phi_{(13)}(G) \in K_{(3,4,4,4)}$ as well. Now, we see that $3$ polynomials $G, \phi_{(23)}(G), \phi_{(13)}(G) \in K_{(3,4,4,4)}$
  are linearly independent as follows.
  For the term $f=s_{1110}^3s_{0111}s_{0003}$ of $G$,
  we find that
  $\phi_{(23)}(f)=s_{1101}^3s_{0111}s_{0030}$ and $\phi_{(13)}(f) =s_{1011}^3s_{0111}s_{0300}$
  are not terms of $G$, and hence
  $f$ is not a term of $\phi_{(23)}(G), \phi_{(13)}(G)$.
  It follows that $G$ is the unique member of
  $\set{G, \phi_{(23)}(G), \phi_{(13)}(G)}$
  having $f$.
  By symmetry,
  $\phi_{\tau}(G)$
  is the unique member
  having $\phi_{\tau}(f)$ as its term, which implies the independence of them.

  Next, we can get the dimension $\kappa_{\beta} := \dim K_{\beta}$ for each $\beta$, for example,
  by calculating the rank of the linear map $\Psi|_{L_{\beta}}$ or 
  counting independent linear relations of the coefficients of a general polynomial in $L_\beta$ induced from Corollary~\ref{coeff-f-in-ker} (see Appendix~\ref{rem:coeff-36-poly} for the details; see also Remark~\ref{rem:proc}). As a result, $\kappa_{(2,4,4,5)} = 1, \kappa_{(3,4,4,4)} = 3, \kappa_{(3,3,4,5)} = 1$,
  and $K_{\beta} = 0$ for any $\beta \notin \mathfrak{S}_{4} \set{(2,4,4,5), (3,4,4,4), (3,3,4,5)}$.
  Thus the equalities~(\ref{eq:FGH-prolong}) hold.
  
  By Proposition~\ref{prop-dsum-beta},
  we have
  the direct sum decomposition~(\ref{eq:FGH-prolong-2}).
  We count the elements of the orbit as
  $\# \mathfrak{S}_{4}\cdot(2,4,4,5) = [\mathfrak{S}_{4}: \lin{(12)}] = 24/2 = 12$,
  where the stabilizer $(\mathfrak{S}_{4})_{(2,4,4,5)} = \lin{(12)}$.
  Similarly,
  $\# \mathfrak{S}_{4}\cdot(3,4,4,4) = [\mathfrak{S}_{4}: \lin{(23), (234)}] = 4$,
  $\# \mathfrak{S}_{4}\cdot(3,3,4,5) = [\mathfrak{S}_{4}: \lin{(01)}] = 12$.
  Hence
  $\dim I(\sigma_4(v_3(\PP^3)))_{5} = 1 \times 12 + 3 \times 4 + 1 \times 12 = 36$.
\end{proof}

Ending this subsection, we would like to add one remark as follows:
\begin{remark}[Approach using the $\GL(V)$-action]
  For a given $\sigma_k(v_d(\PP V))$, instead of the ``torus'' action that we use here, there is another (in fact, a larger) group action on the variety, which is the action by the general linear group $\GL(V)$. For the torus action we have a decomposition of $I(\sigma_4(v_3(\PP^3)))_{5}$ as in (\ref{eq:FGH-prolong-2}), but it is irreducible under the $\GL(V)$-action (see also Corollary \ref{res_Schur} for the $\GL(V)$-invariant form of the free resolution of the whole ideal). Using the $\GL(V)$-action and highest weight theory, one may simplify some parts of the computations in the paper. For instance, it is well-known that every irreducible $\GL(V)$-module has one highest weight vector and ``that irreducible $\GL(V)$-module is in the ideal $I(\sigma_k(v_d(\PP V)))$'' is equivalent to ``its highest weight vector belongs to the ideal'' (see e.g., \cite{LM04, La}). Thus, to show that a given secant variety $\sigma_k(v_d(\PP V))$ is $k$-minimal or $k$-del Pezzo as in $\S$\ref{pf:thm1_and_cor2} and in $\S$\ref{sec:some-small-degrees}, one may try to do this just by checking the membership problem only for those polynomials corresponding to highest weight vectors of relevant isotypic components.

However, in this paper we keep the analysis using the torus action and matrices with bases, because it is also worthwhile to see how the concrete polynomials can be constructed from submatrices of flattenings and combinatorially optimized prolongations.
\end{remark}

\subsection{Higher secant varieties of minimal degrees and next-to-minimal degrees}
\label{sec:secant-vari-minim}

Let $X \subset \PN$ be a non-degenerate projective variety, and let $e=\codim(\sigma_k(X), \PN)$.
In the sense of \cite{CR,CK},
we say $\sigma_k(X)$ is
\emph{a $k$-secant variety of minimal degree}
if $\deg \sigma_k(X) = \binom{e+k}{k}$.
The next-to-extremal case in such degrees
has been investigated in \cite{CK},
where
$\sigma_k(X)$ is said to be
\emph{a del Pezzo $k$-secant variety}
if $\deg \sigma_k(X) = \binom{e+k}{k}+\binom{e+k-1}{k-1}$
and the sectional genus
$\pi(\sigma_k(X)) = (k-1) \Bigl(\binom{e+k}{k}+\binom{e+k-1}{k-1}\Bigr) +1$.

In this setting, by \cite[Corollary 2.15]{CK} and \cite[Theorem~1.1, (3b)]{CK}, it holds
\begin{gather}\label{eq:dimI-min}
  \dim I(\sigma_k(X))_{k+1}
  = \beta_{1,k}(\sigma_k(X)) \leq B_{1,k}^e := \binom{e+k}{1+k},
\end{gather}
and the equality holds
if and only if $\sigma_k(X)$ is a $k$-secant variety of minimal degree.
By \cite[Proposition~6.10]{CK} and \cite[Theorem~1.2, (3b)]{CK},
in the case when
$\sigma_k(X)$ is not a $k$-secant variety of minimal degree,
it holds 
\begin{gather}\label{eq:dimI-dP}
  \dim I(\sigma_k(X))_{k+1} = \beta_{1,k}(\sigma_k(X)) \leq {B_{1,k}'{}^{\kern-1.5ex e}}\, := \binom{e+k}{1+k} - \binom{e+k-2}{k-1},
\end{gather}
and the equality holds
if and only if $\sigma_k(X)$ is a del Pezzo $k$-secant variety.

\subsection{Proof of Theorem~\ref{thm_m1} and Corollary~\ref{res_inv_form}}
\label{pf:thm1_and_cor2}

Now, let us prove our first main result, which deals with the case of $\sigma_4(v_3(\PP^3)) \subset \PP^{19}$.
Note that $I(\sigma_4(v_3(\PP^3)))_j = 0$ for $j \leq 4$ is a well-known fact (for instance, see \cite[Proposition~7.5.1.2]{La} or \cite[Proposition~4.3]{SS}).

\begin{proof}[Proof of Theorem~\ref{thm_m1}]
  In our setting $e=k=4$,
  we have $\binom{e+k}{k} = 70$, $\binom{e+k-1}{k-1} = 35$, and
  ${B_{1,k}'{}^{\kern-1.5ex e}} = 36$.
  From Theorem~\ref{FGH-prolong}, 
  since the equality in (\ref{eq:dimI-dP}) holds
  (i.e., the condition~(3b) of \cite[Theorem~1.2]{CK} holds),
  we conclude that
  $\sigma_4(v_3(\PP^3))$ is a del Pezzo $4$-secant variety,
  where $\deg \sigma_4(v_3(\PP^3)) = 105$ and $\pi(\sigma_4(v_3(\PP^3))) = 316$.
  In addition, it satisfies property $N_{5,3}$, that is to say,
  $\beta_{i,j}(\sigma_4(v_3(\PP^3))) = 0$ for every $i \leq 3$ and every $j \geq 5$
  (see \cite[Theorem~1.2(4) and Definition~2.11]{CK}).
  Since $\beta_{1,j}(\sigma_4(v_3(\PP^3))) = 0$ for $j \geq 5$,
  the ideal $I(\sigma_4(v_3(\PP^3)))$
  has no minimal generator of degree greater than $5$,
  which implies that $I(\sigma_4(v_3(\PP^3)))$ is minimally generated
  by $36$ polynomials of degrees $5$. Finally, the minimal free resolution which is arithmetically Gorenstein of codimension $4$ also can be obtained by the equivalence in \cite[Theorem~1.2]{CK}.
\end{proof}

\begin{remark}\label{symm_Salmon}
  Theorem~\ref{thm_m1} addresses the symmetric version of the ``Salmon problem''. In 2007, E.~Allman posed a problem of computing the ideal of $\sigma_4(Seg(\PP^3\times\PP^3\times\PP^3))$ in $\PP^{63}$ motivated by phylogenetics (see \cite{Allman}). Set-theoretic defining equations were determined independently by Bates and Oeding \cite{BO} and Friedland and Gross \cite{FG12}; these consist of 1728 equations in degree 5, 1000 equations in degree 6, and 8000 equations in degree 9.
  Note that the case of $\sigma_4(Seg(\PP^3\times\PP^3\times\PP^3))$ is neither $4$-secant minimal nor del Pezzo $4$-secant, because $\dim I(\sigma_4(Seg(\PP^3\times\PP^3\times\PP^3)))_5=1728$ was verified (e.g., in \cite{LM04}) and it does not satisfy the condition (\ref{eq:dimI-min}) for $4$-secant minimal or (\ref{eq:dimI-dP}) for del Pezzo $4$-secant.
\end{remark}

One can also write the free resolution of this $\sigma_4(v_3(\PP^3))$ in terms of $\GL(\CC^4)$-modules as follows:

\begin{corollary}\label{res_Schur}
The minimal resolution of $I=I(\sigma_4(v_3(\PP^3)))$ can be represented as
\begin{gather*}
   0 \rightarrow S_{9,9,9,9}V\otimes S(-12) \rightarrow S_{7,5,5,4}V\otimes S(-7) \rightarrow \big(S_{6,5,4,3}V\oplus S_{5,5,4,4}V)\otimes S(-6) \rightarrow S_{5,4,4,2}V\otimes S(-5) \rightarrow  I \rightarrow 0~,
\end{gather*}
where $V=\CC^4$, $S=Sym^\bullet (V)$ and $S_\alpha V$ be the Schur module associated to a partition $\alpha$.
\end{corollary}
\begin{proof}
By Theorem~\ref{thm_m1}, we know that $I$ is a Gorenstein ideal of codimension $4$ and that the minimal free resolution of $I$ is
\begin{gather*}
   0 \rightarrow F_4:=S(-12) \rightarrow F_3:=S(-7)^{36} \rightarrow F_2:=S(-6)^{70} \rightarrow F_1:=S(-5)^{36} \rightarrow I \rightarrow 0~.
\end{gather*}
First of all, by degree consideration, the last free module $F_4=S_{9,9,9,9}V$. For $F_1$, $S_{5,4,4,2}V$ is the only $36$-dimensional irreducible $\GL(V)$-representation in $S_5(S_3 V)$, so $F_1$ is equal to this. And, using the duality of the resolution, $F_3$ should be $(S_{5,4,4,2}V)^\ast\otimes S_{9,9,9,9}V=S_{7,5,5,4}V$. Finally, among the direct summands in $S_{5,4,4,2}V\otimes S_3 V$, we have two candidates for a $70$-dimensional $\GL(V)$-representation; one is $S_{7,4,4,3}V$ and the other is $S_{6,5,4,3}V\oplus S_{5,5,4,4}V$. But the former is not self-dual, so we conclude that the middle $F_2$ is $S_{6,5,4,3}V\oplus S_{5,5,4,4}V$.
\end{proof}

Given a degree $d$-form $\phi$ in $S=Sym^\bullet (V)$, let us denote the image of $\phi$ under the partial polarization map $Sym^d(V)\to Sym^a(V)\otimes Sym^{d-a}(V)$ by $\phi_{a,d-a}$. This can be considered as a linear map $Sym^a(V^\ast)\to Sym^{d-a}(V)$ and is often called as \emph{symmetric flattening} (see e.g., \cite[Section 3.1]{FH} for more details). 

\begin{proof}[Proof of Corollary~\ref{res_inv_form}]
By inheritance \cite[Proposition~2.3.1]{LO}, it is enough to consider the case of $n=3$. By Corollary~\ref{res_Schur}, the module $S_{5,4,4,2}V$ generates the ideal in case of $n=3$.
\end{proof}

\subsection{Another expression of the polynomial $G$ using minimal free resolution}
\label{sec:minim-free-resol}

By Theorem~\ref{thm_m1}, we have the minimal free resolution of $I(\sigma_4(v_3(\PP^3)))$ as follows:
\begin{gather*}
   0 \rightarrow S(-12) \xrightarrow{M_4} S(-7)^{36} \xrightarrow{M_3} S(-6)^{70} \xrightarrow{M_2} S(-5)^{36} \xrightarrow{M_1} I(\sigma_4(v_3(\PP^3))) \rightarrow 0~,
\end{gather*}
with differential matrices $M_i$'s representing homomorphisms of free modules.
Here, both of $M_1$ and $M_4$ consist of $36$ polynomials generating $I(\sigma_4(v_3(\PP^3)))_5$.
Note that $M_3$ is a $70 \times 36$ matrix of linear polynomials such that $M_3\cdot M_4=0$.
An explicit expression of these matrices of the minimal free resolution can be obtained
by computer algebra system (e.g., \cite{DGPS},\cite{SageMath},\cite{M2}).
Then, choosing suitable bases of free modules,
we have a nice expression of $M_3$ (whose submatrix is the skew symmetric matrix $A_1$)
as in (\ref{eq:good-expre-M_3}) below.

\begin{definition}
  Similar to the matrix $A_1$ in Definition~\ref{def-FGH-mat-A1},
  we set $A_2$ to be the following $10 \times 11$ matrix,
  where labels of columns are the same as $A_1$,
  and labels of rows $4446, \dots, 5544$ are new.
\begin{gather*}
  \begin{matrix}
    \ &2445&3534&3543&3453&3444(1)&3354&2544&2454&3345&3444(2)&3435
    \\
    4446&s_{2001}&0&0&0&-s_{1002}&0&0&0&0&s_{1002}&-s_{1011}
    \\
    4455(2)&-s_{2010}&0&0&0&s_{1011}&0&0&0&0&-s_{1011}&s_{1020}
    \\
    4545(2)&s_{2100}&0&0&0&-s_{1101}&0&0&0&0&s_{1101}&-s_{1110}
    \\
    5445&s_{3000}&0&0&0&-s_{2001}&0&0&0&0&s_{2001}&-s_{2010}
    \\
    5346&0&0&0&0&0&0&0&0&-s_{2001}&0&0
    \\
    5355&0&0&0&0&0&0&0&0&s_{2010}&0&0
    \\
    5445&0&0&0&0&0&0&0&0&-s_{2100}&0&0
    \\
    6345&0&0&0&0&0&0&0&0&s_{3000}&0&0
    \\
    5454&0&0&0&-s_{2001}&s_{2010}&-s_{2100}&0&-s_{3000}&0&-2 s_{2010}&0
    \\
    5544&0&-s_{2010}&-s_{2001}&0&-2 s_{2100}&0&-s_{3000}&0&0&2 s_{2100}&0
    \\
  \end{matrix}
\end{gather*}
We set $B_2$ to be the following $10 \times 10$ matrix,
where labels of rows are the same as $A_2$,
and labels of columns $3444(3), \dots, 4344(2)$ are new.
\begin{gather*}
  \begin{matrix}
    \ &3444(3)&4434&4344(1)&4443&4245&4335&5343&5244&5334&4344(2)
    \\
    4446&s_{1002}&-s_{0012}&-s_{0102}&-s_{0003}&-s_{0201}&s_{0111}&0&0&0&0
    \\
    4455(2)&-s_{1011}&s_{0021}&s_{0111}&s_{0012}&s_{0210}&-s_{0120}&0&0&0&0
    \\
    4545(2)&s_{1101}&-s_{0111}&-s_{0201}&-s_{0102}&-s_{0300}&s_{0210}&0&0&0&0
    \\
    5445&s_{2001}&-s_{1011}&-s_{1101}&-s_{1002}&-s_{1200}&s_{1110}&0&0&0&0
    \\
    5346&0&0&0&0&s_{1101}&-s_{1011}&s_{0003}&s_{0102}&-s_{0012}&s_{1002}
    \\
    5355&0&0&0&0&-s_{1110}&s_{1020}&-s_{0012}&-s_{0111}&s_{0021}&-s_{1011}
    \\
    5445&0&0&0&0&s_{1200}&-s_{1110}&s_{0102}&s_{0201}&-s_{0111}&s_{1101}
    \\
    6345&0&0&0&0&-s_{2100}&s_{2010}&-s_{1002}&-s_{1101}&s_{1011}&-s_{2001}
    \\
    5454&-s_{2010}&s_{1020}&s_{1110}&s_{1011}&0&0&-s_{0111}&-s_{0210}&s_{0120}&-s_{1110}
    \\
    5544&s_{2100}&-s_{1110}&-s_{1200}&-s_{1101}&0&0&s_{0201}&s_{0300}&-s_{0210}&s_{1200}
    \\
  \end{matrix}
\end{gather*}
\end{definition}

By choosing bases of $S(-6)^{70}$ and $S(-7)^{36}$,
the $70\times 36$ matrix $M_3$ is expressed as
\begin{gather}\label{eq:good-expre-M_3}
  \begin{bmatrix}
    A_1 & O & O\\
    A_2 & B_2 & O \\
    * & * & *
  \end{bmatrix},
\end{gather}
where $O$ is a zero matrix of suitable size, and $*$ is a matrix of certain entries.
Then $M_3\cdot M_4 = 0$ means that,
for the column vector $v_l$ of $M_3$ with label $l$
and for the polynomial $P_l$ of $M_4$ of weight $l$,
the sum $\sum_{l} v_l P_l = 0$.
In this notation, $P_{2445} = F$, $P_{3345} = H$, and $P_{3444(3)} = G$.

\begin{proposition}\label{prop:G-from-mat}
  Let $\tilde{A}_1$ be
  the $10\times 10$ skew-symmetric matrix
  given by removing the $3435$ column and the $3564$ row from $A_1$,
  and let $c_1$
  be the column vector of size $10$ given by taking the $3435$ column of $A_1$
  and removing $3564$ row.
  Let $\tilde{A}_2$ be
  the $10\times 10$ matrix
  given by removing the $3435$ column from $A_2$,
  let $c_2$
  be the column vector of size $10$ given by taking the $3435$ column of $A_2$,
  and
  let $\tilde{B}_2$ be
  the $10\times 9$ matrix
  given by removing the $3444(3)$ column from $B_2$.
  We set a $20 \times 20$ matrix
  \begin{gather*}
    L =
    \begin{bmatrix}
      \tilde{A}_1 & c_1 & O \\
      \tilde{A}_2 & c_2 & \tilde{B}_2
    \end{bmatrix},
  \end{gather*}
  where $O$ is the $10 \times 9$ zero matrix.
  Then we have
  \begin{gather}\label{eq:G-det-pf}
    G = -\frac{\det(L)}{\Pf(\tilde{A}_1)^3}.
  \end{gather}
\end{proposition}

\begin{proof}
  Let $d_2$
  be the column vector of size $10$ given by taking the $3344(3)$ row of $B_2$.
  We set the following $20 \times 21$ matrix and $20 \times 20$ matrix,
  \begin{gather*}
    X =
    \begin{bmatrix}
      \tilde{A}_1 & c_1 & O & O \\
      \tilde{A}_2 & c_2 & d_2 & \tilde{B}_2
    \end{bmatrix},
    \quad
    N =
    \begin{bmatrix}
      \tilde{A}_1 & O & O \\
      \tilde{A}_2 & d_2 & \tilde{B}_2
    \end{bmatrix}
    =
    \begin{bmatrix}
      \tilde{A}_1 & O \\
      \tilde{A}_2 & B_2
    \end{bmatrix},
  \end{gather*}
  where $\det(N) = \det(\tilde{A}_1)\det(B_2)$.

  Taking
  $M_4'$ to be the 21 polynomials of $M_4$ of weight
  $2445, \dots, 3435, 3444(3), \dots, 4334(2)$,
  we have $X \cdot M_4' = 0$.
  Let $E_{n}$ be the identity matrix of size $n$,
  and let $\check{N}$ be the adjugate matrix of $N$, i.e.,
  $\check{N}\cdot N = \det(N) \cdot E_{20}$,
  which is
  \begin{gather*}
    \begin{bmatrix}
      \check{N}
      \begin{bmatrix}
        \tilde{A}_1 \\
        \tilde{A}_2
      \end{bmatrix}
      &
        \check{N}
        \begin{bmatrix}
          O \\
          d_2
        \end{bmatrix}
      &
        \check{N}
        \begin{bmatrix}
          O \\
          \tilde{B}_2
        \end{bmatrix}
    \end{bmatrix}
    = \
    \begin{matrix}
      \ &&3444(3)&
      \\
      &\det(N) \cdot E_{10} & O & O
      \\
      4446& O & \det(N)& O
      \\
      &O&O&\det(N) \cdot E_{9}\,
    \end{matrix}
  \end{gather*}
  where $4446$ means the label of its row, and $3444(3)$ means the label of its column.
  Then $\check{N}\cdot X$ is equal to
  \begin{gather*}
    \begin{bmatrix}
      \check{N}
      \begin{bmatrix}
        \tilde{A}_1 \\
        \tilde{A}_2
      \end{bmatrix}
      &
        \check{N}
        \begin{bmatrix}
          c_1 \\
          c_2
        \end{bmatrix}
      &
        \check{N}
        \begin{bmatrix}
          O \\
          d_2
        \end{bmatrix}
      &
        \check{N}
        \begin{bmatrix}
          O \\
          \tilde{B}_2
        \end{bmatrix}
    \end{bmatrix}
    = \
    \begin{matrix}
      \ &&3435&3444(3)&
      \\
        &\det(N) \cdot E_{10} & * & O & O
      \\
      4446& O & \det(L) & \det(N)& O
      \\
        &O&*&O&\det(N) \cdot E_{9},
    \end{matrix}
  \end{gather*}
  where note that the 4446 row of the column vector
  $\check{N}\begin{bmatrix} c_1 \\ c_2 \end{bmatrix}$
  is equal to $\det(L)$, because of the definition of cofactors of the matrix.
  Since $\check{N}\cdot X \cdot M_4' = 0$,
  it follows $\det(L) P_{3435} + \det(N) P_{3444(3)} = 0$.
  As in Definition~\ref{def-FH-1}, we have $P_{3435} = \Pf(\tilde{A}_1)$.
  (It is also given as $\phi_{(12)}(H)$.)
  Then
  \begin{gather*}
    G = P_{3444(3)}
    = -\frac{\det(L) \Pf(\tilde{A}_1)}{\det(N)}
    = -\frac{\det(L)}{\Pf(\tilde{A}_1)\det(B_2)}.
  \end{gather*}
  We may compute $\det(B_2) = \Pf(\tilde{A}_1)^2$. Hence the assertion follows.
\end{proof}

\begin{remark}
  In the formula~(\ref{eq:G-det-pf}),
  the computation of the determinant of the $20 \times 20$ matrix $L$ is actually heavy.
  Then, in order to get $G$ from a matrix computation,
  it may be better to construct a smaller matrix, as follows.
  Using the inverse matrix $\tilde{A}_1^{-1}$, we have
  \begin{gather*}
    L
    \begin{bmatrix}
      \tilde{A}_1^{-1} & O\\ 
      O & E_{10}
    \end{bmatrix}
    =
    \begin{bmatrix}
      \tilde{A}_1 &c_1 &O \\
      \tilde{A}_2 &c_2 &\tilde{B}_2
    \end{bmatrix}
    \begin{bmatrix}
      \tilde{A}_1^{-1} & O\\ 
      O & E_{10}
    \end{bmatrix}
    =
    \begin{bmatrix}
      E_{10} &c_1 & O \\
      A_2\tilde{A}_1^{-1} &c_2 &\tilde{B}_2
    \end{bmatrix},
  \end{gather*}
  which is transformed to the following matrix
  \begin{gather*}
    \begin{bmatrix}
      E_{10} & O & O\\
      O &\tilde{c}_2 &\tilde{B}_2
    \end{bmatrix},
  \end{gather*}
  where $\tilde{c}_2 = c_2 - A_2A_1^{-1}c_1$.
  It follows that
  $\det(L)/\det(\tilde{A}_1) = \det \left(\begin{bmatrix} \tilde{c}_2 & \tilde{B}_2 \end{bmatrix} \right)$,
  and hence,
  \begin{gather*}
    G = -\frac{\det \left(\begin{bmatrix} \tilde{c}_2 & \tilde{B}_2 \end{bmatrix} \right)}{\Pf(\tilde{A}_1)}.
  \end{gather*}
  Since $\begin{bmatrix} \tilde{c}_2 & \tilde{B}_2 \end{bmatrix}$
  is a $10 \times 10$ matrix,
  the computation of its determinant is relatively faster,
  though entries of $\tilde{c}_2$ are fractions of polynomials.
\end{remark}

\section{More examples of $k$-secant varieties of relatively small degrees}\label{sec:some-small-degrees}

In this section,
we show that more higher secant varieties $\sigma_k(v_d(\PP^n))$
are \emph{of minimal degree} or \emph{del Pezzo $k$-secant} by calculating
the dimension of each direct summand
$K_{\beta} = \ker(\Psi|_{L_{\beta}})$ of $I(\sigma_k(v_d(\PP^n)))_{k+1}$ in a similar way as in Section \ref{sect_s4(v3(P3))} using Proposition~\ref{prop-dsum-beta} and Corollary~\ref{coeff-f-in-ker}.

To deal with such general $k$-secant varieties of Veronese varieties, we propose the following algorithm:
Recall that $c(f; h)$ denotes the coefficient of a monomial $h \in L_{\beta}$ in $f$ (with divided factorials, as in Corollary~\ref{coeff-f-in-ker}).
To determine the dimension of $K_\beta$, we count how many coefficients $c(f; h)$ of $f \in K_{\beta}$ can be regarded as independent \emph{variables} (or \emph{parameters}) that decide all remaining coefficients.

\begin{procedure}[Computation of $K_\beta$]\label{our_procedure}
  Three sets $\sV, \sD, \sE$ are used for storing coefficients regarded as variables, coefficients depending on variables and their linear equations, and linear equations for variables.

  More precisely, in each step we have (i) $\sV=\{c(f; h_1), \dots, c(f; h_\ell)\}$ as a set of coefficients of $f$ with respect to $h_i \in M(L_{\beta})$, which are regarded as \emph{temporary variables}, (ii) $\sD$ as a set of pairs of a coefficient $c(f; h)$ and a linear equation $g$, where $g$ consists of $c(f; h)$ and variables in $\sV$ (it means $c(f; h)$ depends on variables in $\sV$ under $g=0$), and
  (iii) $\sE$ as a set of linear equations of variables in $\sV$
  (it means that these variables can be also linearly dependent, and the exact rank is determined by solving the system $\sE$).

  We run the following steps in the procedure.
\begin{enumerate}[\;{Step} A:]

\item 
  (A-1) If there exist an $\epsilon \subset \beta$ and a monomial $m \in M(L_{\beta-\epsilon})$ such that
  the left hand side of the equality
  $\sum_{q \in M(L_{\epsilon})} \frac{1}{v(q)}\, c(f; mq) = 0$
  in (\ref{rel_coeff-f-in-ker}) of Corollary~\ref{coeff-f-in-ker}
  consists of variables in $\sV$, coefficients in $\sD$, and \emph{only one new coefficient} $c(f;h')$, then using linear equations in $\sD$ we can write $c(f;h')$ as a linear sum of $c(f; h_1), \dots, c(f; h_\ell)$, the variables in $\sV$. In other words, we can have a linear equation $g'=0$ consisting of $c(f; h_1), \dots, c(f; h_\ell)$ and $c(f;h')$; in this case, let us add the pair ($c(f;h')$, $g'$) to $\sD$ and continue the process.
  \smallskip

  \noindent
  (A-2) If there exist an $\epsilon \subset \beta$ and a monomial $m \in M(L_{\beta-\epsilon})$ such that
  the left hand side of
  $\sum_{q \in M(L_{\epsilon})} \frac{1}{v(q)}\, c(f; mq) = 0$
  consists of
  variables $c(f; h_1), \dots, c(f; h_\ell)$ and coefficients in $\sD$,
  then using linear equations in $\sD$
  we can have a linear equation $g''=0$ consisting of variables $c(f; h_1), \dots, c(f; h_\ell)$;
  in this case, we add $g''$ to $\sE$, the system of linear equations of the variables,
  and continue the process.

\item 
  If there is no more $\epsilon \subset \beta$ corresponding to Step A, then
  we choose a coefficient $c(f; h'')$ not belonging to $\sV$ and $\sD$, and add it to $\sV$ as a new variable. Then we return Step A.

\item
  The procedure is finished when Step A and Step B are checked for every possible $\epsilon \subset \beta$ with $|\epsilon|=2d$ and every $m \in M(L_{\beta-\epsilon})$,
  and there is nothing more to add to $\sV, \sD$ and $\sE$.
  Then, $\kappa_{\beta}=\dim K_{\beta} \leq \#\sV$,
  and the difference $\#\sV - \kappa_{\beta}$ can be computed
  as the rank of the system of linear equations in $\sE$.

  Furthermore, it is also possible to get explicit polynomials which give a basis of $K_{\beta}$ by using the linear equations in $\sD$ and $\sE$.

\end{enumerate}
\end{procedure}

In a typical situation, we may start the procedure with empty $\sV, \sD, \sE$, and first try to find some unbalanced weights $\epsilon\subset\beta$ such as $\epsilon=(2d,0,0), (2d-1,1,0)$ or its permutations, where $M(L_{\epsilon})$ is a singleton $\{q\}$; it follows from Corollary~\ref{coeff-f-in-ker} that $c(f; mq)=0$ for every monomial $m\in  M(L_{\beta-\epsilon})$,
and thus we can add these coefficients to $\sD$ as in (A-1) of Step A.
For an explicit calculation, see Example~\ref{rem:calc-844} below.

It is also possible to start the procedure by pre-specifying some coefficients as variables and taking $\sV$ as the set of them. 
Also, if one only needs to know the dimension of $K_{\beta}$ or its upper bound, then the procedure may be modified to reduce computational resources, for example, by not storing all information to $\sD$ and $\sE$.

\begin{remark}\label{rem:proc}
  We discuss the practical efficiency of our method.
  As in Section~\ref{sect_prolong}, the map $\Psi$ associated with prolongation is just a linear map, and its kernel can theoretically be computed using standard linear algebra. However, the size of the representation matrix grows rapidly as $k,d,n$ increase.
  For the case of the del Pezzo $4$-secant in the previous section, that is $(k,d,n) = (4,3,3)$,
  the domain $\CC[{s_{\alpha}}]_{k+1}$ of $\Psi$ is of dimension $\tbinom{B+k}{k+1} =42504$ with $B = \tbinom{n+d}{d}$ and the target space
  $\bigoplus_{m} \CC[t_0, t_1, \dots, t_n]_{2d} \cdot e(m)$ is of dimension $\tbinom{B+k-2}{k-1} \tbinom{n+2d}{2d} = 129360$.
  It is computationally very hard to determine the matrix rank or the generators of the kernel using standard methods. By contrast, the combinatorial algorithm of Procedure~\ref{our_procedure} enables significantly more efficient computation.

  We provide a Macaulay2 package\footnote{\url{https://github.com/katufurukawa/ProlongationKernels}} to verify these results.
  For $(k,d,n) = (4,3,3)$, running the function of the package on a standard personal computer (with an Intel Core i7-1260P CPU, 12 cores, 16 threads, up to 4.7 GHz in our case) allows one to not only determine the dimension of the prolongation kernel but also obtain the exact $36$ generators in about two seconds.
  In the case of $(k,d,n) = (6,5,2)$, corresponding to Theorem~\ref{case-6,5,2} below, where the domain and target space are of dimensions $888030$ and $3506580$ respectively, the computation of the prolongation kernel and its $15$ generators is completed within about two minutes.
  In the case of $(k,d,n) = (8,4,3)$, Theorem~\ref{case-8,4,3}, for instance the weight subspace $L_{(9,9,9,9)}$ is of dimension $936785$, and one can determine the dimensions of subspaces $K_{(9,9,9,9)}[(01), (23)]$ and $K_{(9,9,9,9)}[-(01), -(23)]$ in about 100 and 50 minutes respectively.
  Our package also provides a way to track all intermediate steps of the procedure (as in Example~\ref{rem:calc-844}).
\end{remark}

\subsection{Del  Pezzo $k$-secant varieties}\label{subsect_dP varieties}

\subsubsection{$\sigma_3(v_4(\PP^2)) \subset \PP^{14}$}

We study the third secant variety $\sigma_3(v_4(\PP^2)) \subset \PP^{14}$,
which is of codimension $e=6$ in $\PP^{14}$.
For $\beta \in \ZZ_{\geq 0}^3$ with $|\beta| = 16$,
we calculate the dimension of each direct summand
$K_{\beta}$ of $I(\sigma_3(v_4(\PP^2)))_4$, as follows.
We have 
$\kappa_{(6,5,5)} = \kappa_{(6,6,4)} = 7$,
$\kappa_{(7,5,4)} = 4$,
$\kappa_{(7,6,3)} = 3$,
$\kappa_{(7,7,2)} = \kappa_{(8,5,3)} = \kappa_{(8,6,2)} = 1$,
$\kappa_{(8,4,4)} = 2$, and other $K_{\beta} = 0$ up to permutations (for instance, we see the calculation $\kappa_{(8,4,4)} = 2$ in Example~\ref{rem:calc-844} below).
Since $\# \mathfrak{S}_3\cdot\beta = [\mathfrak{S}_{3}: (\mathfrak{S}_3)_{\beta}]$
for the stabilizer $(\mathfrak{S}_3)_{\beta}$, it follows from Proposition~\ref{prop-dsum-beta} that 
\begin{gather*}
  \dim I(\sigma_3(v_4(\PP^2)))_4 = 7 \times 3 + 7 \times 3 + 4 \times 6 + 3 \times 6 + 1 \times 3 + 1 \times 6 + 1 \times 6 + 2 \times 3 = 105,
\end{gather*}
which gives the equality in (\ref{eq:dimI-dP}).
Since the condition~(3b) of \cite[Theorem~1.2]{CK} holds,
we can prove 
\begin{theorem}\label{case-3,4,2}
$\sigma_3(v_4(\PP^2))$ is a del Pezzo $4$-secant variety,
and the prime ideal $I(\sigma_3(v_4(\PP^2)))$ is minimally generated by $105$ polynomials of degrees $4$. Furthermore, those $105$ quartic generators can be computed explicitly using Procedure~\ref{our_procedure}.    
\end{theorem}

\begin{remark}
Note that it was previously known in \cite{Sch} and \cite[Theorem~3.2.1(1)]{LO} that
$\sigma_3(v_4(\PP^2))$ can be defined by the $4$-minors of a symmetric flattening matrix $\phi_{2,2}$ only \textit{scheme-theoretically}. Theorem~\ref{case-3,4,2} confirms those 4-minors indeed cut out the variety \textit{ideal-theoretically}.    
\end{remark}

\begin{example}\label{rem:calc-844}
  Let us calculate $\kappa_{(8,4,4)} = 2$. Note that $\dim L_{(8,4,4)}=51$, which can be obtained as regarding non-negative solutions of an integer matrix system.
  We first get that $c(f; h) = 0$ for $11+8+2=21$ monomials $h \in M(L_{(8,4,4)})$
  in the first $3$~steps.
  \medskip

  \noindent{}\emph{Step 1}
  ($11$ monomials):
  \begin{multline*}
    s_{400}^{2} s_{040} s_{004},\
    s_{400}^{2} s_{031} s_{013},\
    s_{400}^{2} s_{022}^{2},\
    s_{400} s_{310} s_{130} s_{004},\
    s_{400} s_{310} s_{121} s_{013},
    s_{400} s_{310} s_{112} s_{022},\\
    s_{400} s_{310} s_{103} s_{031},\
    s_{400} s_{301} s_{130} s_{013},\
    s_{400} s_{301} s_{121} s_{022},\
    s_{400} s_{301} s_{112} s_{031},\
    s_{400} s_{301} s_{103} s_{040},\
  \end{multline*}
  $\epsilon = (8, 0, 0), (8, 0, 0), (8, 0, 0), (7, 1, 0), (7, 1, 0), (7, 1, 0)$,
  $(7, 1, 0), (7, 0, 1), (7, 0, 1)$, $(7, 0, 1), (7, 0, 1)$.

  For example, 
  we have 
  $\frac{1}{2}c(f; s_{400}^{2} s_{040} s_{004}) = 0$,
  $\frac{1}{2}c(f; s_{400}^{2} s_{031} s_{013}) = 0$, and
  $\frac{1}{2}c(f; s_{400}^{2} s_{022}^{2}) = 0$
  by using Corollary~\ref{coeff-f-in-ker} with $\epsilon = (8, 0, 0)$
  and $M(L_{(8,0,0)}) = \set{s_{400}^{2}}$.
  \medskip

  \noindent{}\emph{Step 2}
  ($8$ monomials):
  \begin{multline*}
    s_{310}^{2} s_{220} s_{004},\
    s_{310} s_{301} s_{220} s_{013},\
    s_{400} s_{220}^{2} s_{004},\
    s_{400} s_{220} s_{211} s_{013},\\
    s_{310} s_{301} s_{202} s_{031},\
    s_{301}^{2} s_{202} s_{040},\
    s_{400} s_{211} s_{202} s_{031},\
    s_{400} s_{202}^{2} s_{040},\
  \end{multline*}
  $\epsilon = (5, 3, 0), (5, 3, 0), (4, 4, 0), (4, 3, 1), (5, 0, 3), (5, 0, 3), (4, 1, 3), (4, 0, 4)$.

  Here in Step~2, using monomials which were already given in Step 1,
  we find monomials $h$ satisfying $c(f; h) = 0$.
  For example, 
  we have
  $c(f; s_{310}^{2} s_{220} s_{004}) + c(f; s_{400} s_{310} s_{130} s_{004}) = 0$
  because of $M(L_{(5, 3, 0)}) = \set{s_{310} s_{220}, s_{400} s_{130}}$;
  then, by Step 1, it holds
  $c(f; s_{310}^{2} s_{220} s_{004}) = -c(f; s_{400} s_{310} s_{130} s_{004}) = 0$.
  In the same way,
  we also have
  $c(f; s_{310} s_{301} s_{220} s_{013}) = - c(f; s_{400} s_{301} s_{130} s_{013}) = 0$.
  \medskip

  Next, in Step~3, we run the same process using monomials given in Step 1 and Step 2.
  \medskip

  \noindent{}\emph{Step 3}
  ($2$ monomials):
  $s_{310}^{2} s_{211} s_{013},\
  s_{301}^{2} s_{211} s_{031},\
  $
  $\epsilon = (6, 2, 0), (6, 0, 2)$.\medskip

  After Step 3, even if we use the above new two monomials,
  no more linear relation can be obtained such that some coefficients are zero.
  Then, we specify a coefficient as a variable.
  \medskip

  \noindent{}\emph{Step 4}:
  Add $c(f; s_{310}^{2} s_{202} s_{022})$ as a variable.
  \medskip

  Now we determine $c(f; h)$ for $4$ monomials $h \in M(L_{(8,4,4)})$
  in Step 5 and Step 6.
  \medskip

  \noindent{}\emph{Step 5}
  ($2$ monomials):
  $s_{400} s_{220} s_{202} s_{022},\
  s_{310} s_{301} s_{211} s_{022},\
  $
  $\epsilon = (6, 2, 0), (5, 1, 2)$.

  Here we can use monomials given in Steps 1, 2, 3, and 4.
  For example, 
  we have
  $\frac{1}{2}c(f; s_{310}^{2} s_{202} s_{022}) + c(f; s_{400} s_{220} s_{202} s_{022}) = 0$
  because of $M(L_{(6, 2, 0)}) = \set{s_{310}^2, s_{400} s_{220}}$;
  then $c(f; s_{400} s_{220} s_{202} s_{022})$ is determined as $-\frac{1}{2}c(f; s_{310}^{2} s_{202} s_{022})$.

  In each steps below, we can use monomials given in earlier steps in this way.
  \medskip

  \noindent{}\emph{Step 6}
  ($2$ monomials):
  $s_{400} s_{211}^{2} s_{022},\
  s_{301}^{2} s_{220} s_{022},\
  $
  $\epsilon = (6, 1, 1), (5, 2, 1)$.\medskip

  Similar to Step 4,
  since we cannot determine some coefficients anymore, we specify a coefficient as a variable.
  \medskip

  \noindent{}\emph{Step 7}:
  Add $c(f; s_{310}^{2} s_{121} s_{103})$ as a variable.
  \medskip

  We determine $c(f; h)$ for the remaining monomials $h \in M(L_{(8,4,4)})$
  from Step 8 to Step 14.
  \medskip

  \noindent{}\emph{Step 8}
  ($2$ monomials):
  $s_{400} s_{220} s_{121} s_{103},\
  s_{310}^{2} s_{112}^{2},\
  $
  $\epsilon = (6, 2, 0), (2, 2, 4)$.\medskip

  \noindent{}\emph{Step 9}
  ($3$ monomials):
  \begin{gather*}
    s_{400} s_{220} s_{112}^{2},\
    s_{400} s_{211} s_{130} s_{103},\
    s_{301} s_{220} s_{202} s_{121},\
  \end{gather*}
  $\epsilon = (6, 2, 0), (3, 4, 1), (5, 0, 3)$.\medskip

  \noindent{}\emph{Step 10}
  ($5$ monomials):
  \begin{gather*}
    s_{310} s_{220} s_{211} s_{103},\
    s_{310} s_{301} s_{130} s_{103},\
    s_{301} s_{211} s_{202} s_{130},\
    s_{400} s_{211} s_{121} s_{112},\
    s_{400} s_{202} s_{130} s_{112},\
  \end{gather*}
  $\epsilon = (5, 3, 0), (6, 1, 1), (3, 4, 1), (2, 3, 3), (3, 1, 4)$.\medskip

  \noindent{}\emph{Step 11}
  ($6$ monomials):
  \begin{gather*}
    s_{310} s_{220} s_{202} s_{112},\
    s_{301} s_{220}^{2} s_{103},\
    s_{310} s_{301} s_{121} s_{112},\
    s_{301}^{2} s_{130} s_{112},\
    s_{310} s_{202}^{2} s_{130},\
    s_{400} s_{202} s_{121}^{2},\
  \end{gather*}
  $\epsilon = (5, 3, 0), (4, 4, 0), (6, 1, 1), (6, 0, 2), (5, 1, 2), (2, 4, 2)$.\medskip

  \noindent{}\emph{Step 12}
  ($5$ monomials):
  \begin{gather*}
    s_{220}^{2} s_{202}^{2},\
    s_{310} s_{211} s_{202} s_{121},\
    s_{301} s_{220} s_{211} s_{112},\
    s_{301}^{2} s_{121}^{2},\
    s_{310} s_{211}^{2} s_{112},\
  \end{gather*}
  $\epsilon = (4, 4, 0), (5, 2, 1), (4, 3, 1), (6, 0, 2), (4, 2, 2)$.\medskip

  \noindent{}\emph{Step 13}
  ($2$ monomials):
  $s_{220} s_{211}^{2} s_{202},\
  s_{301} s_{211}^{2} s_{121},\
  $
  $\epsilon = (4, 3, 1), (5, 1, 2)$.\medskip

  \noindent{}\emph{Step 14}
  ($1$ monomials):
  $s_{211}^{4},\
  $
  $\epsilon = (4, 2, 2)$.\medskip

  As a result, all the coefficients of $f$ are written as linear sums of $2$ variables $c(f; s_{310}^{2} s_{202} s_{022})$ and $c(f; s_{310}^{2} s_{121} s_{103})$.
  In this case, we may check that the $2$ variables have no linear relation.
  Hence $\kappa_{(8,4,4)} = 2$.
\end{example}

\subsubsection{$\sigma_6(v_5(\PP^2)) \subset \PP^{20}$}

The sixth secant variety $\sigma_6(v_5(\PP^2)) \subset \PP^{20}$
is of codimension $e=3$.
For $\beta \in \ZZ_{\geq 0}^3$ with $|\beta| = 35$, following Procedure~\ref{our_procedure}  
we calculate the dimension of each direct summand as follows:
$\kappa_{(12,12,11)} = 2$,
$\kappa_{(13,11,11)} = \kappa_{(13,12,10)} = 1$, and other $\kappa_{\beta} = 0$ up to permutations.
Then
\begin{gather*}
  \dim I(\sigma_6(v_5(\PP^2)))_7 = 2 \times 3 + 1 \times 3 + 1 \times 6 = 15.
\end{gather*}
Since the equality holds in (\ref{eq:dimI-dP}),
we can prove 

\begin{theorem}\label{case-6,5,2}
$\sigma_6(v_5(\PP^2))$ is a del Pezzo $6$-secant variety. Its homogeneous prime ideal $I(\sigma_6(v_5(\PP^2)))$ is a Gorenstein ideal of codimension $3$ and minimally generated by these $15$ polynomials of degrees $7$. Furthermore, those $15$ degree $7$ generators can be computed explicitly using Procedure~\ref{our_procedure}.    
\end{theorem}

\begin{remark}
Note that it was previously known that
$\sigma_6(v_5(\PP^2))$ can be defined by $14 \times 14$-subPfaffians of a certain matrix only \emph{scheme-theoretically} (see \cite[Theorem~4.2.7]{LO}).
\end{remark}

\subsection{$k$-secant varieties of minimal degrees}\label{subsect_k-VMD}

\subsubsection{$\sigma_4(v_4(\PP^2)) \subset \PP^{14}$}

By \cite{ElS}, $\sigma_4(v_4(\PP^2)) \subset \PP^{14}$ is of degree $35$.
Hence it is a $4$-secant variety of minimal degree,
and $\dim I(\sigma_4(v_4(\PP^2)))_5 = 21$ as in (\ref{eq:dimI-min}).
Here, we just try to check it using our method.
For $\beta \in \ZZ_{\geq 0}^3$ with $|\beta| = 20$,
we have $\kappa_{(8,8,4)} = \kappa_{(8,7,5)} = 1$ and
$\kappa_{(8,6,6)} = \kappa_{(7,7,6)} = 2$ and other $K_{\beta} = 0$ up to permutations.
It implies
$\dim I(\sigma_4(v_4(\PP^2)))_5 = 1 \times 3 + 1 \times 6 + 2 \times 3 + 2  \times 3 = 21$.

\subsubsection{$\sigma_8(v_4(\PP^3)) \subset \PP^{34}$}

The eighth secant variety $\sigma_8(v_4(\PP^3)) \subset \PP^{34}$ is of codimension $e=3$.
Since $B_{1,8}^3 = 55$, we have
$\dim I(\sigma_8(v_4(\PP^3)))_9 \leq 55$ by (\ref{eq:dimI-min}).
In order to show that the equality holds,
we calculate the dimension of $K_{\beta} \subset I(\sigma_8(v_4(\PP^3)))_9$
for $\beta \in \ZZ_{\geq 0}^{4}$ with $|\beta| = 36$.

Let us consider $\beta = (9,9,9,9)$ and set $K = K_{(9,9,9,9)}$.
By symmetry on $\PP^3$, it holds $\phi_{(01)}(K) = K$ for $(01) \in \mathfrak{S}_4$.
Then, as in Remark~\ref{split-inv-ker-im},
taking $\lambda_{(01)}(f) = f + \phi_{(01)}(f)$,
we have $K \simeq K[(01)] \oplus K[-(01)]$.
On the other hands, 
$K[(01)]$ and $K[-(01)]$
are invariant over $\phi_{(23)}$.
Hence
$K[(01)] \simeq K[(01), (23)] \oplus K[(01), -(23)]$
and
$K[-(01)] \simeq K[-(01), (23)] \oplus K[-(01), -(23)]$.
As a result, $K_{(9,9,9,9)}$ is decomposed as
\begin{gather*}
  K_{(9,9,9,9)}[(01), (23)] \oplus K_{(9,9,9,9)}[-(01), -(23)]
  \oplus K_{(9,9,9,9)}[-(01), (23)] \oplus K_{(9,9,9,9)}[(01), -(23)].
\end{gather*}
(Note that $K_{(9,9,9,9)}[-(01), (23)] \simeq K_{(9,9,9,9)}[(01), -(23)]$.)

By computation, 
$\kappa_{(9,9,9,9)}[(01), (23)] = 2$ and $\kappa_{(9,9,9,9)}[-(01), -(23)] = 1$;
hence $\kappa_{(9,9,9,9)} \geq 3$.
In the same way, we have
$\kappa_{(10,9,9,8)} \geq \kappa_{(10,9,9,8)}[(12)] = 2$,
$\kappa_{(10,10,8,8)} \geq  \kappa_{(10,10,8,8)}[(01), (23)] = 2$,
$\kappa_{(10,10,9,7)} \geq \kappa_{(10,10,9,7)}[(01)] = 1$, and
$\kappa_{(10, 10, 10,6)} \geq \kappa_{(10, 10, 10,6)}[(01)] = 1$.
Then
$\dim I(\sigma_8(v_4(\PP^3)))_9 \geq 3 + 2 \times 12 + 2 \times 6 + 1 \times 12 + 1 \times 4 = 55$.
Hence the equality holds,
and it also implies that the equality holds for every intermediate formula.

As a result, we obtain
\begin{theorem}\label{case-8,4,3}
$\sigma_8(v_4(\PP^3))$ is an $8$-secant variety of minimal degree $165$. Its homogeneous prime ideal $I(\sigma_8(v_4(\PP^3)))$ is minimally generated by
these $55$ polynomials of degrees $9$. Furthermore, those $55$ degree $9$ generators can be computed explicitly using Procedure~\ref{our_procedure}.     
\end{theorem}

\begin{remark} In \cite{CK2} it is known that for the case of $k$-secant varieties of minimal degree (such as $\sigma_8(v_4(\PP^3))$) the ideal can have a \textit{determinantal} presentation as well.
\end{remark}

\begin{remark}[Affirmative cases related a question in \cite{BGL}]\label{relation_with_BGL} Finally, we would like to remark that the cases verified to be a $k$-secant variety of minimal degree or a del Pezzo $k$-secant variety in Table \ref{tab:types-of-some-secant} provides affirmative instances related to Question 1.6 in \cite{BGL}, where the authors asked that, for a given $k>0$, if there does exist a lower bound $d_0=d_0(n,k)$ such that the ideal $I(\sigma_k(v_d(\PP^n)))$ is generated in degree $k+1$ for an infinitely many (or even all) $d\ge d_0$ or not.
\end{remark}

\subsection{Other higher secant varieties}\label{sec:other-higher-secant}

In the case that $(k,d,n)$ is one of the seven
\[
  (2,4,2), (5,5,2), (8,6,2), (3,3,3), (7,4,3), (6,3,4), (9,3,5),
\]
computing the dimension of $I(\sigma_k(v_d(\PP^n)))_{k+1}$
and showing that $\dim I(\sigma_k(v_d(\PP^n)))_{k+1}$ is strictly less than the bound in (\ref{eq:dimI-dP}),
we can prove that the $\sigma_k(v_d(\PP^n))$ is 
\emph{neither} a $k$-secant variety of minimal degree \emph{nor} a $k$-del Pezzo as shown in Table \ref{tab:types-of-some-secant}. 

In following subsections, we present some details of computations for
$(k,d,n) = (7,4,3), (9,3,5)$ as examples.

\subsubsection{$\sigma_7(v_4(\PP^3)) \subset \PP^{34}$}

The seventh secant variety $\sigma_7(v_4(\PP^3)) \subset \PP^{34}$
is of codimension $e=7$.
Consider $K_{\beta} \subset I(\sigma_7(v_4(\PP^3)))_8$
with $\beta \in \ZZ_{\geq 0}^{4}$ and $|\beta| = 32$.
We compute
$\kappa_{(8,8,8,8)} = \kappa_{(8,8,8,8)}[(01), (23)] + \kappa_{(8,8,8,8)}[(01), -(23)] + \kappa_{(8,8,8,8)}[-(01), (23)] + \kappa_{(8,8,8,8)}[-(01), -(23)] = 12+4+4+6 = 26$. Similarly,
$\kappa_{(9, 8, 8, 7)} = 17$,
$\kappa_{(9, 9, 7, 7)} = 12$,
$\kappa_{(9, 9, 8, 6)} = 10$,
$\kappa_{(9, 9, 9, 5)} = 5$,
$\kappa_{(10, 7, 7, 8)} = \kappa_{(10, 8, 8, 6)} = 7$,
$\kappa_{(10, 10, 6, 6)} = 2$,
$\kappa_{(10, 9, 9, 4)} = \kappa_{(10, 10, 7, 5)} = \kappa_{(10, 10, 8, 4)} = 1$,
$\kappa_{(10, 9, 7, 6)} = 4$,
$\kappa_{(10, 9, 8, 5)} = 3$, and other $K_{\beta} = 0$ up to permutations.
As a result,
$\dim I(\sigma_7(v_4(\PP^3)))_{8} = 26 \times 1 + 17 \times 12 + 12 \times 6 + 10 \times 12 + 5 \times 4 + 7 \times 12 + 7 \times 12 + 2 \times 6 + 1 \times 12 + 1 \times  12 + 1 \times 12 + 4 \times 24 + 3 \times 24 = 826$.
Since ${B_{1,7}'{}^{\kern-1.5ex 7}} = 2079$, the equality does not hold in (\ref{eq:dimI-dP}).

\subsubsection{$\sigma_9(v_3(\PP^5)) \subset \PP^{55}$}

We consider the ninth secant variety $\sigma_9(v_3(\PP^5)) \subset \PP^{55}$,
which is of codimension $e=2$.
Since $B_{1,9}^2 = 11$, we have
$\dim I(\sigma_9(v_3(\PP^5)))_{10} \leq 11$ in (\ref{eq:dimI-min}).
By computation,
\begin{gather*}
  \kappa_{(5,5,5,5,5,5)} = \kappa_{(5,5,5,5,5,5)}[-(01), -(23), -(45)] = 1.
\end{gather*}
On the other hands, $\dim K_{\beta} = 0$ for $|\beta| = 30$ and $\beta \neq (5,5,5,5,5,5)$
as follows.
First,
except when
$\beta$ is one of
$
(30, 0, 0, 0, 0, 0),
(25, 1, 1, 1, 1, 1),
(20, 2, 2, 2, 2, 2),
(15, 3, 3, 3, 3, 3),
(10, 4, 4, 4, 4, 4),
(6, 6, 6, 6, 6, 0)
$
up to permutations,
it holds that
$\dim K_{\beta} = 0$ since $\#\mathfrak{S}_6\beta \geq 11$
(For example, $\#\mathfrak{S}_6\beta = 6!/4! = 30$ for $\beta = (4, 5, 5, 5, 5, 6)$.)
For each exceptional $\beta$, we have $\#\mathfrak{S}_6\beta =6$;
then it is also impossible to be $\dim K_{\beta} > 0$
since $\dim I(\sigma_9(v_3(\PP^5)))_{10} \leq 11$.

As a result, we have
$\dim I(\sigma_9(v_3(\PP^5)))_{10} = 1 < {B_{1,9}'{}^{\kern-1.5ex 2}} = 2$
in (\ref{eq:dimI-dP}).

\appendix

\section{Computing $\dim K_\beta$ in $I(\sigma_4(v_3(\PP^3)))_5$}\label{rem:coeff-36-poly}

In this appendix, we
show how to evaluate
$\kappa_{\beta} = \dim K_{\beta}$
in the proof of Theorem~\ref{FGH-prolong} in \textsection{}\ref{sub_sect_FGH} as following the procedure in \textsection{}\ref{sec:some-small-degrees} (in fact, the procedure presented in Procedure \ref{our_procedure} is slightly modified here to make computations shorter).
As a result, we can prove
$\kappa_{(2,4,4,5)}, \kappa_{(3,3,4,5)} \leq 1$, $\kappa_{(3,4,4,4)} \leq 3$,
and $K_{\beta} = 0$ for $\beta \notin \mathfrak{S}_{4} \set{(2,4,4,5), (3,4,4,4), (3,3,4,5)}$.
\medskip

  \begin{inparaenum}[(a)]
  \item 
    (a-1)
    Corollary~\ref{coeff-f-in-ker}
    provides an information about coefficients of polynomials of $I(\sigma_4(v_3(\PP^3)))_{5}$.
    For instance, for a generator $F$ of $I(\sigma_4(v_3(\PP^3)))_{5}$ given in
    Definition~\ref{def-FH-1},
    we consider
    the coefficient $c(F; s_{\alpha_1}^{i_1} \dotsm\, s_{\alpha_{\mu}}^{i_{\mu}})$
    with respect to 
    $\frac{1}{i_1 ! \dotsm i_{\mu}!}\, s_{\alpha_1}^{i_1} \dotsm\, s_{\alpha_{\mu}}^{i_{\mu}}$
    in the case of $\beta=(2,4,4,5)$.
    For $\epsilon = (0,0,1,5)$,
    we have $M(L_{(0,0,1,5)}) = \set{s_{0012}s_{0003}}$.
    Then $c(F; m \cdot s_{0012}s_{0003}) = 0$ for any $m \in ML_{(2,4,3,0)}$. 
    Similarly, $c(F; m \cdot s_{0102}s_{0003}) = 0$ for any $m \in M(L_{(2,3,4,0)})$,
    and $c(F; m \cdot s_{1002}s_{0003}) = 0$ for any $m \in ML_{(1,4,4,0)}$. 
    For $\epsilon = (0,0,2,4)$, we have $M(L_{(0,0,2,4)}) = \set{s_{0021}s_{0003}, s_{0012}^2}$.
    Then $c(F; m \cdot s_{0021}s_{0003}) = - \frac{1}{2} c(F; m \cdot s_{0012}^2)$
    for any $m \in M(L_{(2,4,2,1)})$.
    \medskip

    (a-2) Moreover, using Corollary~\ref{coeff-f-in-ker},
    we may get an upper bound of $\dim K_{\beta}$
    for each direct summand $K_{\beta}$ of $I(\sigma_4(v_3(\PP^3)))_{5}$.
    We set $\tilde\rho: \CC[\set{s_\alpha}] \rightarrow R_{\xi}$ to be the natural homomorphism induced by the ring isomorphism $\rho$ in (\ref{eq:qring_iso_local}).
    Then $\tilde\rho(s_{\alpha})$ is obtained by (\ref{eq:qring_iso_local__s})
    for
    \begin{gather*}
      \alpha = (2,0,0,1), (2,0,1,0), (2,1,0,0), (3,0,0,0),
    \end{gather*}
    and $\tilde\rho(s_{\alpha}) = s_{\alpha}$ otherwise.
    Let $f \in K_{\beta}$.
    Since $f \in I(\sigma_4(v_3(\PP^3)))$, we have $\tilde\rho(f) = 0$.
    It means that,
    if the coefficient $c(f; s_{\alpha} \cdot h)$ is determined (resp., equal to zero) for any $h \in M(L_{\beta-\alpha})$ and $\alpha = (2,0,0,1), (2,0,1,0), (2,1,0,0), (3,0,0,0)$,
    then $f$ itself is determined (resp., equal to zero).
    We use this implication to shorten the computations for this specific case as follows.
    \medskip

  \item
    We may check $K_{\beta} = 0$ if
    $\beta \notin \mathfrak{S}_{4} \set{(2,4,4,5), (4,3,3,5), (3,4,4,4)}$.
    For example, let us take $\beta = (2,3,5,5)$ and $f \in K_{(2,3,5,5)}$.
    In this case, we do not need to consider $\alpha = (3,0,0,0)$,
    because of $\alpha \not\subset \beta$.
    \medskip

    (b-1)
    For $\alpha = (2,1,0,0)$, we have $c(f; s_{2100} \cdot h) = 0$
    for all $16$ monomials $h \in M(L_{(0,2,5,5)})$
    in the following $3$ steps.
    \medskip

    \noindent{}\emph{Step 1}. The coefficient $c(f; s_{2100} \cdot h) = 0$ if $h$ is one of the $5$ monomials:
    \begin{gather*}
      s_{0120} s_{0102} s_{0030} s_{0003},\
      s_{0201} s_{0030} s_{0021} s_{0003},\
      s_{0102}^{2} s_{0030} s_{0021},\
      s_{0210} s_{0030} s_{0012} s_{0003},\
      s_{0120}^{2} s_{0012} s_{0003},
    \end{gather*}
    where we use $\epsilon = (0, 1, 5, 0), (0, 0, 5, 1), (0, 0, 5, 1), (0, 0, 1, 5), (0, 0, 1, 5)$,
    respectively.
    
    For example, 
    applying Corollary~\ref{coeff-f-in-ker} to $\epsilon = (0, 1, 5, 0)$
    and $M(L_{(0,1,5,0)}) = \set{s_{0120}s_{0030}}$,
    we have $m=s_{2100}s_{0102}s_{0003}$ so that $c(f; s_{2100} s_{0120} s_{0102} s_{0030} s_{0003}) = 0$.
    In the same way, if $\epsilon = (0, 0, 5, 1)$, $M(L_{(0, 0, 5, 1)}) = \set{s_{0030} s_{0021}}$ and $m=s_{2100}s_{0201}s_{0003}, s_{2100} s_{0102}^{2}$,
    we have 
    
    $c(f; s_{2100} s_{0201} s_{0030} s_{0021} s_{0003}) = 0$
    and $c(f; s_{2100} s_{0102}^{2} s_{0030} s_{0021}) = 0$.
    \medskip

    \noindent{}\emph{Step 2}. The coefficient $c(f; s_{2100} \cdot h) = 0$
    if $h$ is one of the $8$ monomials:
    \begin{multline*}
      s_{0111}^{2} s_{0030} s_{0003},\
      s_{0120} s_{0111} s_{0021} s_{0003},\
      s_{0210} s_{0021}^{2} s_{0003},\
      s_{0111} s_{0102} s_{0030} s_{0012},\\
      s_{0210} s_{0021} s_{0012}^{2},\
      s_{0201} s_{0021}^{2} s_{0012},\
      s_{0120} s_{0102} s_{0021} s_{0012},\
      s_{0201} s_{0030} s_{0012}^{2},
    \end{multline*}
    where we use
    $\epsilon = (0, 2, 2, 2), (0, 1, 3, 2), (0, 0, 4, 2), (0, 1, 2, 3), (0, 0, 3, 3), (0, 0, 3, 3), (0, 0, 3, 3), (0, 0, 2, 4)$, respectively.

    Here in Step~2, using monomials given in Step 1,
    we find monomials $h$ satisfying $c(f; s_{2100} \cdot h) = 0$.
    For example, for $\epsilon=(0, 2, 2, 2)$ and
    $M(L_{(0, 2, 2, 2)})
    = \set{s_{0210}s_{0012}, s_{0201}s_{0021}, s_{0120}s_{0102}, s_{0111}^2}$,
    we have
    \begin{multline*}
      c(f; s_{2100} s_{0210} s_{0030} s_{0012} s_{0003})
      + c(f; s_{2100} s_{0201} s_{0030} s_{0021} s_{0003}) + \\
      c(f; s_{2100} s_{0120} s_{0102} s_{0030} s_{0003})
      + \frac{1}{2} c(f; s_{2100} s_{0111}^{2} s_{0030} s_{0003}) = 0
    \end{multline*}
    Then, Step 1 implies $c(f; s_{2100} s_{0111}^{2} s_{0030} s_{0003}) = 0$.
    In each steps below, we can use monomials given in earlier steps in this way.
    \medskip

    \noindent{}\emph{Step 3}. The coefficient $c(f; s_{2100} \cdot h) = 0$
    if $h$ is one of the $3$ monomials:
    \begin{gather*}
      s_{0120} s_{0111} s_{0012}^{2},\
      s_{0111}^{2} s_{0021} s_{0012},\
      s_{0111} s_{0102} s_{0021}^{2},
    \end{gather*}
    where we use
    $\epsilon = (0, 2, 3, 1), (0, 2, 2, 2), (0, 1, 3, 2)$, respectively.
    For example, for $\epsilon = (0, 2, 3, 1)$ and
    $M(L_{(0, 2, 3, 1)}) = \set{s_{0210} s_{0021}, s_{0120} s_{0111}, s_{0201} s_{0030}}$ , we have
    \begin{gather*}
      c(f; s_{2100} s_{0210} s_{0021} s_{0012}^{2}) + c(f; s_{2100} s_{0120} s_{0111} s_{0012}^{2}) + c(f; s_{2100} s_{0201} s_{0030} s_{0012}^{2}) = 0.
    \end{gather*}
    Then Step 2 implies $c(f; s_{2100} s_{0120} s_{0111} s_{0012}^{2}) = 0$.
    \medskip

    (b-2)
    Next, in the same way,
    in the case of $\alpha=(2,0,1,0)$, we have $c(f; s_{2010} \cdot h) = 0$
    for all $h \in M(L_{(0,3,4,5)})$, consisting of the following $21$ monomials in $4$ steps.
    \medskip

    \noindent{}\emph{Step 1} ($4$ monomials):
    \begin{gather*}
      s_{0210} s_{0102} s_{0030} s_{0003},\
      s_{0120}^{2} s_{0102} s_{0003},\
      s_{0300} s_{0030} s_{0012} s_{0003},\
      s_{0210} s_{0120} s_{0012} s_{0003},
    \end{gather*}
    $\epsilon = (0, 1, 0, 5), (0, 1, 0, 5), (0, 0, 1, 5), (0, 0, 1, 5)$.
    \medskip

    \noindent{}\emph{Step 2} ($6$ monomials):
    \begin{multline*}
      s_{0201} s_{0111} s_{0030} s_{0003},\
      s_{0210} s_{0111} s_{0021} s_{0003},\
      s_{0300} s_{0021}^{2} s_{0003},\\
      s_{0201} s_{0102} s_{0030} s_{0012},\
      s_{0300} s_{0021} s_{0012}^{2},\
      s_{0210} s_{0102} s_{0021} s_{0012},
    \end{multline*}
    $\epsilon = (0, 3, 1, 2), (0, 1, 3, 2), (0, 0, 4, 2), (0, 3, 0, 3), (0, 0, 3, 3), (0, 0, 3, 3)$.
    \medskip

    \noindent{}\emph{Step 3} ($7$ monomials):
    \begin{multline*}
      s_{0201} s_{0120} s_{0021} s_{0003},\
      s_{0120} s_{0111}^{2} s_{0003},\
      s_{0120} s_{0111} s_{0102} s_{0012},\\
      s_{0201} s_{0111} s_{0021} s_{0012},\
      s_{0201} s_{0102} s_{0021}^{2},\
      s_{0111} s_{0102}^{2} s_{0030},\
      s_{0210} s_{0111} s_{0012}^{2},
    \end{multline*}
    $\epsilon = (0, 3, 2, 1), (0, 2, 3, 1), (0, 2, 3, 1), (0, 3, 1, 2), (0, 0, 4, 2), (0, 2, 1, 3), (0, 1, 2, 3)$.
    \medskip

    \noindent{}\emph{Step 4} ($4$ monomials):
    \begin{gather*}
      s_{0201} s_{0120} s_{0012}^{2},\
      s_{0120} s_{0102}^{2} s_{0021},\
      s_{0111}^{3} s_{0012},\
      s_{0111}^{2} s_{0102} s_{0021},
    \end{gather*}
    $\epsilon = (0, 3, 2, 1), (0, 1, 4, 1), (0, 2, 2, 2), (0, 1, 3, 2)$.
    \medskip

    (b-3) By symmetry, it also holds that
    $c(f; s_{2001} \cdot h) = 0$ for all $h \in M(L_{(0,3,5,4)})$.
    As a summary of (b-1), (b-2), and (b-3), from the discussion in (a-2),
    the polynomial $f = 0$; therefore $K_{(2,3,5,5)} = 0$.
    \medskip

  \item
    We show $\kappa_{(2,4,4,5)} \leq 1$.
    (c-1)
    For $\alpha = (2,1,0,0)$,
    We have $c(f; s_{2100} \cdot h) = 0$
    for all $h \in M(L_{(0,3,4,5)})$, consisting of the $21$ monomials in $4$ steps,
    which is the same as (b-2).
    By symmetry, for $\alpha = (2,0,1,0)$, it holds $c(f; s_{2010} \cdot h) = 0$
    for all $h \in M(L_{(0,4,3,5)})$.
    \medskip

    (c-2) We consider the case of $\alpha = (2,0,0,1)$.
    If we regard $c(f; s_{2001} s_{0210}^{2} s_{0021} s_{0003})$ as a variable,
    then
    $c(f; s_{2001} \cdot h)$ is determined
    for any $h \in M(L_{(0,4,4,4)})$, consisting of $24$ monomials and the monomial $s_{0210}^{2} s_{0021} s_{0003}$; it follows from the same computation for the Aronhold invariant in Example~\ref{eg_Aronhold}.
    \medskip

    (c-3) As a result, if one coefficient is regarded as a variable,
    then the polynomial $f$ is determined
    as we discussed in (a-2).
    It means that
    $\kappa_{(2,4,4,5)} \leq 1$.
    \medskip

  \item
    We show $\dim K_{(3,4,4,4)} \leq 3$
    by checking that,
    if the three coefficients
    \begin{gather}\label{eq:3444-3coef}
      c(f; s_{2001} s_{1200} s_{0210} s_{0030} s_{0003}),\;
      c(f; s_{2100} s_{1020} s_{0300} s_{0021} s_{0003}),\;
      c(f; s_{2010} s_{1002} s_{0300} s_{0102} s_{0030})
    \end{gather}
    are regarded as variables,
    then $f \in K_{(3,4,4,4)}$ is determined.
    \medskip

    (d-1)
    If we regard the three coefficients (\ref{eq:3444-3coef}) as variables,
    then 
    \begin{gather}\label{eq:3444-3coef-B}
      c(f; s_{3000} s_{0210}^{2} s_{0021} s_{0003}),\
      c(f; s_{2100} s_{1200} s_{0111} s_{0030} s_{0003}),\
      c(f; s_{2100} s_{1110} s_{0201} s_{0030} s_{0003}),
    \end{gather}
    are determined;
    in fact, we determine $c(f; g)$
    for the following $8$ monomials $g$ in $4$ steps:
    \medskip

    \noindent{}\emph{Step 1}:
    \begin{gather*}
      s_{2001} s_{1110} s_{0300} s_{0030} s_{0003},\
      s_{2100} s_{1011} s_{0300} s_{0030} s_{0003},\
      s_{2010} s_{1101} s_{0300} s_{0030} s_{0003},\
    \end{gather*}
    $\epsilon = (1, 4, 1, 0), (1, 0, 4, 1), (1, 1, 0, 4)$.
    \medskip

    For example,
    for $\epsilon = (1, 4, 1, 0)$ in Step 1 and
    $M(L_{(1, 4, 1, 0)}) = \set{s_{1200} s_{0210}, s_{1110} s_{0300}}$,
    we have
    \begin{gather*}
      c(f; s_{2001} s_{1200} s_{0210} s_{0030} s_{0003}) + c(f; s_{2001} s_{1110} s_{0300} s_{0030} s_{0003}) = 0.
    \end{gather*}
    Since $c(f; s_{2001} s_{1200} s_{0210} s_{0030} s_{0003})$ is regarded as a variable
    in (\ref{eq:3444-3coef}),
    $c(f; s_{2001} s_{1110} s_{0300} s_{0030} s_{0003})$ is determined
    \medskip

    \noindent{}\emph{Step 2}:
    $s_{2010} s_{1200} s_{0201} s_{0030} s_{0003}$,
    $s_{3000} s_{0300} s_{0111} s_{0030} s_{0003}$,
    $\epsilon = (1, 4, 0, 1), (3, 1, 1, 1)$.
    \medskip

    \noindent{}\emph{Step 3}:
    $s_{2100} s_{1200} s_{0111} s_{0030} s_{0003}$,
    $s_{3000} s_{0210} s_{0201} s_{0030} s_{0003}$,
    $s_{3000} s_{0300} s_{0120} s_{0021} s_{0003}$,

    $\epsilon = (3, 3, 0, 0), (0,4,1,1), (0, 1, 4, 1)$.
    \medskip

    \noindent{}\emph{Step 4}:
    $s_{2100} s_{1110} s_{0201} s_{0030} s_{0003}$,
    $s_{3000} s_{0210}^{2} s_{0021} s_{0003}$,
    $\epsilon = (3, 2, 1, 0), (0, 4, 2, 0)$.
    \medskip

    (d-2) For $\alpha  = (3,0,0,0)$,
    since $c(f; s_{3000} s_{0210}^{2} s_{0021} s_{0003})$ is determined
    in (\ref{eq:3444-3coef-B}),
    we may also determine $c(f; s_{3000} \cdot h)$
    for all $h \in M(L_{(0,4,4,4)})$,
    consisting of the $24$ monomials and $s_{0210}^{2} s_{0021} s_{0003}$
    is the same way as (c-2).
    \medskip

    (d-3) We consider the case of $\alpha = (2,1,0,0)$.
    Since the three coefficients
    \begin{gather*}
      c(f; s_{2100} s_{1020} s_{0300} s_{0021} s_{0003}),\
      c(f; s_{2100} s_{1200} s_{0111} s_{0030} s_{0003}),\
      c(f; s_{2100} s_{1110} s_{0201} s_{0030} s_{0003}),
    \end{gather*}
    are determined
    in (\ref{eq:3444-3coef}) and (\ref{eq:3444-3coef-B}),
    we also determine $c(f; s_{2100} \cdot h)$
    for all $h \in M(L_{(1,3,4,4)})$,
    consisting of the following $49$ monomials (and the above $3$ polynomials)
    in 6~steps.
    \smallskip

    \noindent{}\emph{Step 1}
    ($9$ monomials):
    \begin{multline*}
      s_{1020} s_{0201} s_{0120} s_{0003},\
      s_{1011} s_{0300} s_{0030} s_{0003},\
      s_{1200} s_{0120} s_{0021} s_{0003},\
      s_{1020} s_{0201} s_{0102} s_{0021},\\
      s_{1200} s_{0111} s_{0021} s_{0012},\
      s_{1110} s_{0201} s_{0021} s_{0012},\
      s_{1110} s_{0102}^{2} s_{0030},\
      s_{1200} s_{0102} s_{0030} s_{0012},\
      s_{1020} s_{0300} s_{0012}^{2},\
    \end{multline*}
    $\epsilon = (1, 1, 4, 0), (1, 0, 4, 1), (0, 1, 4, 1), (0, 3, 0, 3), (0, 0, 3, 3), (0, 0, 3, 3), (0, 2, 0, 4), (0, 1, 1, 4), (0, 0, 2, 4)$.
    \smallskip

    \noindent{}\emph{Step 2}
    ($11$ monomials):
    \begin{multline*}
      s_{1110} s_{0210} s_{0021} s_{0003},\
      s_{1011} s_{0210} s_{0120} s_{0003},\
      s_{1020} s_{0120} s_{0102}^{2},\
      s_{1101} s_{0210} s_{0030} s_{0003},\\
      s_{1020} s_{0210} s_{0111} s_{0003},\
      s_{1011} s_{0201} s_{0102} s_{0030},\
      s_{1200} s_{0120} s_{0012}^{2},\
      s_{1200} s_{0102} s_{0021}^{2},\\
      s_{1020} s_{0201} s_{0111} s_{0012},\
      s_{1011} s_{0300} s_{0021} s_{0012},\
      s_{1002} s_{0300} s_{0030} s_{0012},\
    \end{multline*}
    $\epsilon = (1, 3, 2, 0), (0, 3, 3, 0), (1, 1, 4, 0), (1, 3, 1, 1), (0, 3, 2, 1), (1, 0, 4, 1), (0, 1, 3, 2), (0, 0, 4, 2), (0, 1, 2, 3)$, $(0, 0, 3, 3), (1, 0, 1, 4)$.
    \smallskip

    \noindent{}\emph{Step 3}
    ($11$ monomials):
    \begin{multline*}
      s_{1110} s_{0210} s_{0012}^{2},\
      s_{1110} s_{0120} s_{0111} s_{0003},\
      s_{1002} s_{0210} s_{0120} s_{0012},\
      s_{1101} s_{0120}^{2} s_{0003},\\
      s_{1101} s_{0210} s_{0021} s_{0012},\
      s_{1101} s_{0201} s_{0030} s_{0012},\
      s_{1020} s_{0210} s_{0102} s_{0012},\
      s_{1002} s_{0300} s_{0021}^{2},\\
      s_{1110} s_{0111} s_{0102} s_{0021},\
      s_{1020} s_{0111}^{2} s_{0102},\
      s_{1002} s_{0210} s_{0102} s_{0030},\
    \end{multline*}
    $\epsilon = (1, 3, 2, 0), (1, 2, 3, 0), (0, 3, 3, 0), (0, 2, 4, 0), (1, 3, 1, 1), (1, 3, 0, 2), (0, 3, 1, 2), (0, 0, 4, 2), (0, 2, 1, 3)$, $(0, 2, 1, 3), (1, 1, 0, 4)$.
    \smallskip

    \noindent{}\emph{Step 4}
    ($11$ monomials):
    \begin{multline*}
      s_{1110} s_{0120} s_{0102} s_{0012},\
      s_{1002} s_{0120}^{2} s_{0102},\
      s_{1011} s_{0201} s_{0120} s_{0012},\
      s_{1101} s_{0120} s_{0111} s_{0012},\\
      s_{1101} s_{0201} s_{0021}^{2},\
      s_{1101} s_{0111} s_{0102} s_{0030},\
      s_{1002} s_{0201} s_{0111} s_{0030},\
      s_{1011} s_{0210} s_{0111} s_{0012},\\
      s_{1011} s_{0210} s_{0102} s_{0021},\
      s_{1002} s_{0210} s_{0111} s_{0021},\
      s_{1110} s_{0111}^{2} s_{0012},\
    \end{multline*}
    $\epsilon = (1, 2, 3, 0), (0, 2, 4, 0), (1, 1, 3, 1), (0, 2, 3, 1), (1, 3, 0, 2), (1, 2, 1, 2), (0, 3, 1, 2), (1, 1, 2, 2), (1, 0, 3, 2)$, $(0, 1, 3, 2), (0, 1, 2, 3)$.
    \smallskip

    \noindent{}\emph{Step 5}
    ($6$ monomials):
    \begin{multline*}
      s_{1101} s_{0120} s_{0102} s_{0021},\
      s_{1002} s_{0201} s_{0120} s_{0021},\
      s_{1011} s_{0120} s_{0111} s_{0102},\\
      s_{1002} s_{0120} s_{0111}^{2},\
      s_{1011} s_{0201} s_{0111} s_{0021},\
      s_{1101} s_{0111}^{2} s_{0021},\
    \end{multline*}
    $\epsilon = (1, 2, 2, 1), (0, 3, 2, 1), (1, 1, 3, 1), (0, 2, 3, 1), (0, 3, 1, 2), (0, 1, 3, 2)$.
    \smallskip

    \noindent{}\emph{Step 6}
    ($1$ monomial):
    $s_{1011} s_{0111}^{3},\ $
    $\epsilon = (1, 1, 2, 2)$.
    \medskip

    (d-4)
    In the same way, we can determine $c(f; s_{2010} \cdot h)$
    for all $h \in M(L_{(1,4,3,4)})$
    and
    $c(f; s_{2001} \cdot h)$
    for all $h \in M(L_{(1,4,4,3)})$.
    Therefore $f$ is determined if we regard three coefficients (\ref{eq:3444-3coef}) as variables.
    It implies $\kappa_{(3,4,4,4)} \leq 3$.
    \medskip

  \item 
    Similarly, we may have $\kappa_{(3,3,4,5)} \leq 1$.

  \end{inparaenum}
\bigskip


\bigskip




\begin{thebibliography}{99}

\bibitem{AH} J. Alexander and A. Hirschowitz, \emph{Polynomial interpolation in several variables}, J. Algebraic Geom. 4 (1995), 201--222.

\bibitem{Allman} E. Allman, \emph{Open problem: determine the ideal defining $Sec_{4}(\PP^3\times\PP^3\times\PP^3)$}, available online at \url{https://eallman.github.io/papers/salmonPrize.pdf}, 2007.

\bibitem{BO} D. J. Bates and L. Oeding, \emph{Toward a Salmon conjecture}, Exp. Math. 20 (2011), 358--370.

\bibitem{BE} D. A. Buchsbaum and D. Eisenbud, \emph{Algebra structures for finite free resolutions, and some structure theorems for ideals of codimension 3}, Amer. J. Math. 99 (1977), 447--485.

\bibitem{BGL} J. Buczy{\'n}ski, A. Ginensky and J. M. Landsberg, \emph{Determinantal equations for secant varieties and the Eisenbud-Koh-Stillman conjecture}, J. Lond. Math. Soc. (2) 88 (2013), 1--24.

\bibitem{CGNW} T. Chmiel, L. Guerrieri, X. Ni and J. Weyman, \emph{Structure theorems for Gorenstein ideals of codimension four with small number of generators}, preprint, \url{https://arxiv.org/abs/2503.08813}.

\bibitem{CK} J. Choe and S. Kwak, \emph{A matryoshka structure of higher secant varieties and the generalized Bronowski's conjecture}, Adv. Math. 406 (2022), 1--45.

\bibitem{CK2} J. Choe and S. Kwak, \emph{Determinantal characterization of higher secant varieties of minimal degree}, to appear in `Algebraic Geometry', \url{https://arxiv.org/abs/2207.06851}.

\bibitem{CR} C. Ciliberto and F. Russo, \emph{Varieties with minimal secant degree and linear systems of maximal dimension on surfaces}, Adv. Math. 200 (2006), 1--50.

\bibitem{DGPS} W. Decker, G.-M. Greuel, G. Pfister and H. Sch{\"o}nemann, ``Singular, A Computer Algebra System for Polynomial Computations'', \url{https://www.singular.uni-kl.de}.

\bibitem{ElS} G. Ellingsrud and S. A. Str{\o}mme, \emph{Bott's formula and enumerative geometry}, J. Amer. Math. Soc. 9 (1996), 175--193.

\bibitem{FG12} S. Friedland and E. Gross, \emph{A proof of the set-theoretic version of the Salmon conjecture}, J. Algebra 356 (2012), 374--379.

\bibitem{FH} K. Furukawa and K. Han, \emph{On the singular loci of higher secant varieties of Veronese embeddings}, J. Reine Angew. Math. (Crelle’s Journal) 824 (2025), 203--251.

\bibitem{M2} D. R. Grayson and M. E. Stillman, ``Macaulay2, A Software System for Research in Algebraic Geometry'', \url{http://www.math.uiuc.edu/Macaulay2/}.

\bibitem{HK} K. Han and S. Kwak, \emph{Sharp bounds for higher linear syzygies and classifications of projective varieties}, Math. Ann. 361 (2015), 535--561.

\bibitem{HT} J. Harris and L. Tu, \emph{On symmetric and skew-symmetric determinantal varieties}, Topology 23 (1984), 71--84.

\bibitem{IK} A. Iarrobino and V. Kanev, ``Power Sums, Gorenstein Algebras, and Determinantal Loci'', Lecture Notes in Mathematics 1721, Springer-Verlag, Berlin, 1999 (Appendix C by A. Iarrobino and S. L. Kleiman).

\bibitem{Kan} V. Kanev, \emph{Chordal varieties of Veronese varieties and catalecticant matrices}, J. Math. Sci. (N. Y.) 94 (1999), 1114--1125.

\bibitem{KM} A. R. Kustin and M. Miller, \emph{Algebra structures on minimal resolutions of Gorenstein rings of embedding codimension four}, Math. Z. 173 (1980), 171--184.

\bibitem{La} J. M. Landsberg, ``Tensors: Geometry and Applications'', Graduate Studies in Mathematics 128, American Mathematical Society, Providence, RI, 2012.

\bibitem{LM04} J. M. Landsberg and L. Manivel, \emph{On the ideals of secant varieties of Segre varieties}, Found. Comput. Math. 4 (2004), 397--422.

\bibitem{LO} J. M. Landsberg and G. Ottaviani, \emph{Equations for secant varieties of Veronese and other varieties}, Ann. Mat. Pura Appl. (4) 192 (2013), 569--606.

\bibitem{Ott} G. Ottaviani, \emph{An invariant regarding Waring's problem for cubic polynomials}, Nagoya Math. J. 193 (2009), 95--110.

\bibitem{R15} M. Reid, \emph{Gorenstein in codimension 4: the general structure theory}, Adv. Stud. Pure Math. (2015), 201--227.

\bibitem{SageMath} The Sage Developers, ``SageMath, the Sage Mathematics Software System'', \url{https://www.sagemath.org}.

\bibitem{Sch} F.-O. Schreyer, \emph{Geometry and algebra of prime Fano 3-folds of genus 12}, Compos. Math. 127 (2001), 297--319.

\bibitem{SS} J. Sidman and S. Sullivant, \emph{Prolongations and computational algebra}, Canad. J. Math. 61 (2009), 930--949.

\bibitem{W} J. Weyman, \emph{Higher structure theorems for codimension four Gorenstein ideals}, preprint, \url{https://alg.matinf.uj.edu.pl/preprints.html}.

\end{thebibliography}
\end{document}